\documentclass{amsart}
\usepackage{amssymb}
\usepackage{amsfonts}
\usepackage{amssymb}
\usepackage{amsmath}
\usepackage{amsthm}
\usepackage{enumerate}
\usepackage{tabularx}
\usepackage{centernot}
\usepackage{mathtools}
\usepackage{stmaryrd}
\usepackage{amsthm,amssymb}
\usepackage{etoolbox}
\usepackage{url}
\usepackage{tikz}
\usepackage{amssymb}
\usetikzlibrary{matrix}
\usepackage{tikz-cd}
\usepackage{tikz}
\usepackage{marginnote}
\definecolor{mygray}{gray}{0.85}
\usepackage[backgroundcolor=mygray,colorinlistoftodos,prependcaption,textsize=small]{todonotes}
\usepackage{xargs}                      

\renewcommand{\leq}{\leqslant}
\renewcommand{\geq}{\geqslant}

\makeatletter
\def\subsection{\@startsection{subsection}{3}%
  \z@{.5\linespacing\@plus.7\linespacing}{.3\linespacing}%
  {\bfseries\centering}}
\makeatother

\makeatletter
\def\subsubsection{\@startsection{subsubsection}{3}%
  \z@{.5\linespacing\@plus.7\linespacing}{.3\linespacing}%
  {\centering}}
\makeatother

\makeatletter
\def\myfnt{\ifx\protect\@typeset@protect\expandafter\footnote\else\expandafter\@gobble\fi}
\makeatother

\newtheorem{theorem}{Theorem}[section]

\newtheorem{corollary}[theorem]{Corollary}
\newtheorem{definition}[theorem]{Definition}
\newtheorem{lemma}[theorem]{Lemma}
\newtheorem{proposition}[theorem]{Proposition}
\newtheorem{example}[theorem]{Example}

\newtheorem{observation}[theorem]{Observation}
\newtheorem{fact}[theorem]{Fact}

\newtheorem{remark}[theorem]{Remark}
\newtheorem{notation}[theorem]{Notation}
\newtheorem{context}[theorem]{Context}

\newtheorem{construction}[theorem]{Construction}
\newtheorem{convention}[theorem]{Convention}

\newtheorem*{theorem1}{Theorem 1.1}

\newtheorem*{theorem5}{Theorem 1.5}

\newcounter{claimcounter}
\numberwithin{claimcounter}{theorem}
\newenvironment{claim}{\stepcounter{claimcounter}{\noindent {\underline{\em Claim \theclaimcounter}.}}}{}

\newcommand{\pureindep}[1][]{%
  \mathrel{
    \mathop{
      \vcenter{
        \hbox{\oalign{\noalign{\kern-.3ex}\hfil$\vert$\hfil\cr
              \noalign{\kern-.7ex}
              $\smile$\cr\noalign{\kern-.3ex}}}
      }
    }\displaylimits_{#1}
  }
}

\begin{document}

\begin{abstract} We prove that the theory of open projective planes is complete and strictly stable, and infer from this that Marshall Hall's free projective planes $(\pi^n : 4 \leq n \leq \omega)$ are all elementary equivalent and that their common theory is strictly stable and decidable, being in fact the theory of open projective planes. We further characterize the elementary substructure relation in the class of open projective planes, and show in particular that $(\pi^n : 4 \leq n \leq \omega)$ is an elementary chain. 
We then prove that the theory of open projective planes does not have a prime model, that it has elimination of quantifiers down to Boolean combinations of existential formulas, and that it is not model complete.
%
Finally, we characterize the forking independence relation in models of the theory and prove that the $\pi^n$'s ($4 \leq n \leq \omega)$ are strongly type-homogeneous.
\end{abstract}

\title{First-Order Model Theory of Free Projective Planes}
\thanks{The second author was supported by European Research Council grant 338821.}

\author{Tapani Hyttinen}
\address{Department of Mathematics and Statistics,  University of Helsinki, Finland}

\author{Gianluca Paolini}
\address{Department of Mathematics ``Giuseppe Peano'', University of Torino, Italy.}

\date{\today}
\maketitle

\section{Introduction}

	Free constructions in incidence geometry \cite{handbook} are a crucial tool in proving existence results. This is the case for example for the theory of generalized $n$-gons, see e.g. \cite{tits}, but also in many other contexts (see \cite{funk1, funk2}). The first example of free construction was introduced in the context of projective planes by Marshall Hall in his famous paper \cite{hall_proj}. Also in \cite{hall_proj}, as the most canonical example of free construction, Hall introduced certain specific objects: the {\em free projective planes} $\pi^n$. In this paper we will study the first-order model theory of the free projective planes.
	

	A part from the pioneering studies of Hall, the free projective planes received the attention of eminent geometers such as Barlotti \cite{barlotti}, Dembowski \cite{dem}, and Hughes and Piper  \cite[Chapter 7]{piper}, and of many other scholars, see e.g. \cite{ditor, ellers, iden, sieben, sandler}. 

	Already in the original article of Hall \cite{hall_proj} many important results on free projective planes were proved; among them that free projective planes of different rank are non-isomorphic and that finitely generated subplanes of free planes are free -- this last result was later improved by Kopeikina \cite{kope} dispensing with finite generation.
	
	Again in \cite{hall_proj}, Hall isolated a property of projective planes defined fobidding\footnote{A theme certainly dear to model-theorists, see e.g. \cite{cherlin} and references there.} certain finite subconfigurations (and so a first-order property) and called the projective planes satisfying this property {\em open} (cf. Definition \ref{open}). He then proved that every free projective plane is open, and that finitely generated open projective planes are free, leaving open the question of existence of a non-finitely generated open projective plane which is not free. This was settled in the negative by Kopeikina in \cite{kope}, providing a countable counterexample, and later improved by Kelly \cite{kelly, kelly_article}, which showed that there are in fact $\aleph_0$-many countable counterexamples. 
	
	The main result of our paper is the following axiomatization result, which confirms the deep foresight of Hall on the subject:
	
	\begin{theorem}[Main Theorem]\label{th_complete} The theory $T$ of open projective planes is complete.
\end{theorem}
	
	The second crucial notion in our paper is the notion of {\em HF-constructibility}, which connects with the already mentioned notion of openess. The well-founded version of the notion of HF-constructibility was introduced by Siebenmann in \cite{sieben} (see also \cite{ditor, ellers}), where he declares that a (projective) plane $B$ is (well-foundedly) HF-constructible over a plane $A$ if there is an element-by-element well-founded construction of $B$ over $A$ adding at most two incidences at a time (cf. Definition~\ref{HF_construct_def}). Also in \cite{sieben} Siebenmann proved the surprising result that the countable projective planes well-foundedly HF-constructible from $\emptyset$ are {\em exactly} the free projective planes.
	
	In our paper we introduce\footnote{A forerunner of this notion was introduced by the authors in \cite{paolini&hyttinen}.} a natural but crucial variant of this notion which dispenses with the assumption of well-foundedness (cf. Definition \ref{def_nwf_cons}). With this notion at hand, we were able to prove the following crucial characterization:
	
	\begin{theorem}\label{theorem_characterization} A projective plane is open iff it is HF-constructible over the $\emptyset$.
\end{theorem}

	Apart from its model-theoretic usefulness (which will be clear later), Theorem~\ref{theorem_characterization} provides a characterization of the open projective planes which explains exactly which among the countable open projective planes are the free projective planes, i.e. the ones where the HF-construction can \mbox{be taken to be {\em well-founded}.}
	

\smallskip

	We then prove:

	\begin{theorem}\label{th_elem_substr} If $A$ and $B$ are open projective planes and $A \subseteq B$, then $A$ is elementary in $B$ if and only if $B$ is HF-constructible over $A$. Furthermore, if $A$ and $B$ are free projective planes and $A$ is finitely generated, then $A$ is elementary in $B$ if and only if $B$ is well-foundedly\footnote{In \cite{kope} (in Russian) Kopeikina extended Hall's terminology introducing a notion of free product of projective planes and thus a notion of free factor of projective planes. With respect to this termonology (cf. also \cite[pg. 380]{sandler} (in English)) we have that if $A$ and $B$ are free projective planes and $A$ is finitely generated, then $A$ is elementary in $B$ if and only if $A$ is a free factor of $B$.} HF-constructible over $A$.
\end{theorem}

	From the theorems above we infer the following crucial corollary on free projective planes, which was {\em our} original motivation for this study.

	\begin{corollary}[Main Corollary]\label{corollary_free} The free projective planes $(\pi^n : 4 \leq n \leq \omega)$ are all elementary equivalent, and they form an elementary chain with respect to the natural embedding mapping $\pi^n_0$ into $\pi^m_0$, for $4 \leq n \leq m \leq \omega$. Their common theory is the theory of open projective planes, and thus it is decidable.
\end{corollary}

	We wish to observe that the question of elementary equivalence of the free projective planes and of the decidability of their theory(ies) was posed as an open question by Shirshov and Nikitin in their book \cite{russian_book}\footnote{We thank Aleksander Iwanow for telling us about this.} (pg. 68, 
last sentence of Section 16).

\smallskip

	We continued our study of the theory of open projective planes focusing on questions of definability, algebraicity, and homogeneity, proving the following theorems.

	\begin{theorem}\label{th_prime_model} The theory $T$ of open projective planes does {\em not} have a prime model.
\end{theorem}

	\begin{theorem}\label{elim_quantifiers} In  $T$ every formula is equivalent to a Boolean combinations of \mbox{$\exists$-formulas} and $T$ is {\em not} model complete \mbox{(so it does not have quantifier elimination).} 
\end{theorem}

	\begin{theorem}\label{th_acl} In models of $T$ algebraic closure can be described (see Corollary~\ref{cor_char_alg_cl}). Further, for every $M \models T$ such that $\pi^4 \preccurlyeq M$ we have that $acl_M \neq dcl_M$.
\end{theorem}

	\begin{theorem}\label{th_type_homogeneity} The free projective plane $\pi^\omega$ is strongly type-homogeneous, i.e. for every tuple $\bar{a}, \bar{b}$ in $\pi^\omega$ and finite set of parameter $A$ in $\pi^\omega$, $\bar{a}$ and $\bar{b}$ have the same type over $A$ if and only if there is $f \in Aut(\pi^\omega)$ mapping $\bar{a}$ to $\bar{b}$ and fixing $A$ pointwise. Furthermore, the same holds for the free projective planes of finite rank if we require in addition that the algebraic closure of $\bar{a}$ over $A$ (equiv. of $\bar{b}$) is non-degenerate.
\end{theorem}

	Although our focus in the present study was {\em not} on the stability theory of open projective planes, the techniques of proof of Theorems~\ref{th_complete} and \ref{th_elem_substr} allowed us to characterize completely the stability class of the theory of open projective planes:

	\begin{theorem}\label{th_strictly_stable} The theory $T$ of open projective planes is strictly stable. Furthermore, for every infinite cardinality $\kappa$ there are $2^\kappa$ non-isomorphic open projective planes of power $\kappa$.
\end{theorem}

	Recall that previous to our work it was only known that there are $\aleph_0$-many countable non-free open projective planes \cite{kelly_article}. With our methods we were  also able to prove that the natural notion of amalgamation between free projective planes (cf. Definition \ref{def_canonical_amalgam}) corresponds exactly with the notion of independence witnessing the stability of our theory $T$ (i.e. the non-forking independence relation).

	\begin{theorem}\label{th_forking} Let $\mathfrak{M}$ be the monster model of $T$, and $A, B, C \subseteq \mathfrak{M}$. Then $B \pureindep[A] C$~(in the sense of non-forking) if and only if  $acl(ABC)$ is the {\em canonical amalgam} of $acl(AB)$ and $acl(AC)$ over $acl(A)$ (cf. Definition \ref{def_canonical_amalgam}).
\end{theorem}
	
	As mentioned at the beginning, notions of free and open object appear also in many other contexts in combinatorial geometry, most notably in the theory of generalized $n$-gons (for a general study of these phenomena see \cite{funk1, funk2}). We believe that behind our solutions to the main problems faced here there is a whole theory yet to be discovered, on which we intend to return \mbox{on a more general future work.}
	
	We conclude this introduction mentioning other works on incidence structures and model theory. Recently, there have been a number of papers where free constructions (and similia) played an important role, see in particular \cite{muller, muller_tent, tent1, tent2}. These papers, which are certainly interesting, do not seem to have any direct connection with our work, which relies on very different techniques and ideas, most notably the newly introduced notion of HF-construction mentioned several times above. In particular, it was suggested to us that the results from \cite{muller, muller_tent} could help us in our model-theoretic analysis of free projective planes, in particular with respect to the question of stability. 
	However we do not think that the methods from \cite{muller, muller_tent} are of particular relevance to the question of stability of $T$. In fact, by Proposition~\ref{no_disjoint_amalgamation}, the class of finitely generated open projective planes is not closed under free amalgamation, and, by Proposition~\ref{acl_neq_dcl}, all the free projective planes are such that $acl \neq dcl$, while $acl = dcl$ in all the structures admitting a stationary independence relation (which is the framework of \cite{muller}), as observed explicitly in \cite[Fact~3.4]{muller}. It follows readily that the stationary
independence studied in \cite{muller} does {\em not} correspond to non-forking in open projective planes.

	
	The structure of the paper is as follows: in Section~\ref{sec_state_art} we introduce the main definitions and overview the relevant state of the art; in Section~\ref{sec_HF} we introduce the crucial notion of HF-constructibility and prove Theorem~\ref{theorem_characterization}; in Section~\ref{sec_axiomatization} we prove Theorem~\ref{th_complete}; in Section~\ref{sec_elem_sbs} we prove Theorem~\ref{th_elem_substr}; in Section~\ref{sec_prime_models} we prove Theorem~\ref{th_prime_model}; in Section~\ref{sec_alg_cl} we prove Theorems~\ref{elim_quantifiers}~and~\ref{th_acl}; in Section~\ref{sec_homogeneo} we prove  Theorem~\ref{th_type_homogeneity}; in Sections~\ref{sec_unsuper}~and~\ref{sec_stability} we prove Theorems~\ref{th_strictly_stable}~and~Theorem~\ref{th_forking}.
	

\section{State of the Art}\label{sec_state_art}

	\begin{definition}[{\cite{hall_proj}}]\label{def_plane} A {\em partial plane} is a system of points and lines satisfying:
	\begin{enumerate}[(A)]
	\item through any two distinct points $p$ and $p'$ there is at most one line $p \vee p'$;
	\item any two distinct lines $\ell$ and $\ell'$ intersect in at most one point $\ell \wedge \ell'$.
\end{enumerate}
We say that a partial plane is a {\em projective plane} if in (A)-(B) above we replace ``at most'' with ``exactly one''. We say that a projective plane is non-degenerate if it contains a quadrangle, i.e. four points such that no three of them are collinear.
\end{definition}

	\begin{convention}\label{deg_conv} All the projective planes considered in this paper will be assumed to be non-degenerate (cf. Definition \ref{def_plane}). In specific claims we might diverge from this convention; when so, we will specify this divergence explicitly.
\end{convention}

	\begin{definition}\label{subconf} Let $P$ be a partial plane.
	\begin{enumerate}[(1)]
	\item We say that the partial plane $P'$ is a {\em subconfiguration} of $P$ if $P' \subseteq P$, points of $P'$ are points of $P$, lines of $P'$ are lines of $P$, and, for $p$ and $\ell$ in $P'$, $p$ in incident with $\ell$ in $P'$ if and only if $p$ in incident with $\ell$ in $P$.
	\item We say that $P'$ is a {\em closed} subconfiguration of $P$ if $P'$ is a subconfiguration of $P$ and $P'$ is closed under intersection of lines and join of points.
	\end{enumerate}
\end{definition}
	
	\begin{definition}\label{subplane} Let $P$ be a projective plane and $A \subseteq P$.
	\begin{enumerate}[(1)]
	\item We denote by $\langle A \rangle_P$ the smallest closed subconfiguration of $P$ containing $A$, and call it the closed subconfiguration generated by $A$ in $P$.
	\item If $A$ contains a quadrangle, then $\langle A \rangle_P$ is a projective plane. In this case we refer to $\langle A \rangle_P$ as the projective subplane (or simply subplane) \mbox{generated by $A$ in $P$.}
	\end{enumerate}
\end{definition}

	\begin{remark}\label{remark_nondeg} In Definition~\ref{subplane}, notice that $A$ contains a quadrangle iff $\langle A \rangle_P$ contains a quadrangle iff $\langle A \rangle_P$ is a projective plane (cf. Convention~\ref{deg_conv}).
\end{remark}

	\begin{definition}[{\cite{hall_proj}}]\label{free_extension} Given a partial plane $P$ we define a chain of partial planes $(P_n : n < \omega)$, by induction on $n < \omega$, as follows:
\newline $n = 0)$. Let $P_n = P$.
\newline $n = 2k +1)$. For every pair of distinct points $p, p' \in P_{2k}$ not joined by a line add a new line $p \vee p'$ to $P_{2k}$ incident \mbox{with only $p$ and $p'$. Let $P_n$ be the resulting plane.}
\newline $n = 2k >0)$. For every pair of parallel lines $\ell, \ell' \in P_{2k-1}$ add a new point $\ell \wedge \ell'$ to $P_{2k-1}$ incident \mbox{with only $\ell$ and $\ell'$. Let $P_n$ be the resulting plane.}
\newline We define the {\em free projective extension} of $P$ to be $F(P) : = \bigcup_{n < \omega} P_n$.
\end{definition}

	\begin{notation}\label{pi_n} Given $4 \leq n \leq \omega$, we let $\pi_0^n$ be the partial plane consisting of a line $\ell$, $n-2$ points on $\ell$ and $2$ points off of $\ell$, and we let $\pi^n = F(\pi^n_0)$ (cf. Definition~\ref{free_extension}). We refer to the plane $\pi^n$, for $4 \leq n \leq \omega$, as the free projective plane of rank $n$. We say that a plane is free if it is isomorphic to $\pi^n$ for some $4 \leq n \leq \omega$.
\end{notation}

	\begin{fact}[{\cite[Theorems 4.5 and 4.12]{hall_proj}}]\label{fact_noniso} Let $4 \leq n < m \leq \omega$. Then:
\begin{enumerate}[(1)]
	\item $\pi^n$ is not isomorphic to $\pi^m$;
	\item $\pi^n$ contains a copy of $\pi^m$.
\end{enumerate}
\end{fact}

	\begin{definition}\label{open} Let $P$ be a partial plane.  We say that $P$ is {\em open} if there is no {\em finite} subconfiguration $A$ of $P$ such that every element of $A$ is incident with at least three elements of $A$.
\end{definition}

\begin{fact}\label{facts_free_planes} Let $P$ be a countable projective plane. Then (recalling Definitions \ref{subplane} and \ref{pi_n}) we have:
	\begin{enumerate}[(1)]
	\item if $P$ is free and $P'$ is a subplane of $P$, then $P'$ is free (cf. \cite[Theorem I]{sieben});
	\item\label{item3} if $P$ is free, then $P$ is open (cf. \cite[Theorem 4.8]{hall_proj} and \cite[Theorem 2]{sieben});
	\item if $P$ is open and finitely generated, then $P$ is free (cf. \cite[Theorem 4.8]{hall_proj}); 
	\item $P$ is open if and only if every finitely generated subplane of $P$ is free (cf. \cite{dem}).
\end{enumerate}
\end{fact}

	
	\begin{definition}\label{HF_def} Let $P$ be a partial plane and $P+x$ a partial plane containing $P$ such that $x \notin P$ and $P+x = P \cup \{ x \}$. We say that $P+x$ is a hyper-free (abbreviated as HF) one-point extension of $P$ if $x$ is incident with at most two elements of $P$. We say that $P+x$ is of type $i$, for $i = 0, 1, 2$, if in $P+x$ the element $x$ is incident with exactly $i$ elements of $P$. We denote this type as $t(P+x/P)$.
\end{definition}

\begin{definition}\label{HF_construct_def} Let $Q$ and $P$ be partial planes. We say that $P$ is well-foundedly HF-constructible\footnote{In standard references (cf. e.g. \cite{sieben}) these constructions are referred to simply as HF-constructions (without the specification ``well-founded''), the reason for our choice of terminology is because of our more general Definition~\ref{def_nwf_cons}, see also Remark~\ref{remark_well_founded}}  from $Q$ (or over $Q$), denoted as $Q \leq^*_{HF} P$, if there is an ordinal $\alpha$ and a sequence $(P_\beta)_{\beta < \alpha}$ of partial planes such that:
	\begin{enumerate}[(1)]
	\item $P_0 = Q$;
	\item if $\beta = \gamma+1$, then $P_{\beta}$ is a hyper-free one-point extension of $P_\gamma$ (cf. Def.~\ref{HF_def});
	\item if $\beta$ is limit, then $P_\beta = \bigcup_{\gamma < \beta} P_\gamma$;
	\item $P = \bigcup_{\beta < \alpha} P_\beta$.
\end{enumerate}
We say in addition that $P$ is F-constructible from $Q$ if in the sequence $(P_\beta)_{\beta < \alpha}$ we have that $t(P_{\beta+1}/P_\beta) = 2$ (cf. Definition \ref{HF_def}), for every $\beta < \alpha$.
\end{definition}

	\begin{fact}[{\cite[Lemma 1]{sieben}}]\label{open_implies_constr} If $P$ is a finite open partial plane, then $P$ is (well-foundedly) HF-constructible from $\emptyset$ (cf. Definition \ref{HF_construct_def}). 
\end{fact}


%

	\begin{fact}[{\cite[Main Theorem]{sieben} and \cite[Chapter~1]{kelly}}]\label{sieben_fact} Let $P$ be a countable projective plane. Then $P$ is a free projective plane if and only if there is a countable ordinal $\alpha$ and a sequence of partial planes $(P_\beta)_{\beta < \alpha}$, such that:
	\begin{enumerate}[(1)]
	\item $P_0 = \emptyset$;
	\item if $\beta = \gamma+1$, then $P_{\beta}$ is a hyper-free one-point extension of $P_\gamma$ (cf. Def.~\ref{HF_def});
	\item if $\beta$ is limit, then $P_\beta = \bigcup_{\gamma < \beta} P_\gamma$;
	\item $P = \bigcup_{\beta < \alpha} P_\beta$.
\end{enumerate}
Furthermore, if $(P_\beta)_{\beta < \alpha}$ is as above, then the rank of the free plane $P$ is: $$(\sum_{\beta < \alpha} 2 - t(P_{\beta+1}/P_\beta)) - 4.$$
\end{fact}

\section{HF-Orderings}\label{sec_HF}

	\begin{definition}\label{def_directed_graph} A directed graph is a pair $(V, R)$ such that $V$ is a set and $R$ is a collection of ordered pairs from $V$ such that if $(a, b) \in R$, then $a \neq b$ and $(b, a) \notin R$.
\end{definition}

	\begin{definition}\label{def_path} Let $(V, R)$ be a directed graph (a.k.a. digraph) and $a \neq b \in V$. A (directed) path $\pi$ from $a$ to $b$ is a sequence $(a_0, ..., a_n)$ of elements from $V$ such that $n > 0$, $a_0 = a$, $a_n = b$ and $R(a_i, a_{i+1})$ for every $i = 0, ..., n-1$. Given a directed path $\pi = (a_0, ..., a_n)$ from $a$ to $b$ we let the {\em length} of $\pi$ to be $n$. The {\em distance} $d_R(a, b)$ between $a$ and $b$ in $(V, R)$ is the length of the shortest path from $a$ to $b$ (where by convention we set this to be $\infty$ when there is no such path).
\end{definition}

	\begin{definition}\label{HF-digraph} Let $A \subseteq B$ be partial planes (in particular $A$ can be $\emptyset$), and $R$ a directed graph structure on $B$. We say that $(B, R)$ is an HF-digraph over $A$ when:
	\begin{enumerate}[(1)]
	\item if $a \in A$, then for every $b \in B$ we have that $B \models \neg R(b, a)$;
	\item for $a \in B - A$ and $b \in B$, $a$ is incident with $b$ if and only if $R(a, b)$ or $R(b, a)$;
	\item for every $b \in B$, $|\{ a \in B : R(a, b) \}| \leq 2$.
\end{enumerate}
\end{definition}

	\begin{definition}\label{def_nwf_cons} \begin{enumerate}[(1)]
	\item Let $A \subseteq B$ be partial planes (in particular $A$ can be $\emptyset$), we say that $B$ is HF-constructible (resp. F-constructible) from (or over) $A$ if there is a linear ordering $(B - A, <)$ such that for every $b \in B - A$ there are at most two (resp. exactly two) elements of $B$ such that they are incident with $b$ and either from $A$ or from $B-A$ and $<$-smaller than $b$. 
	\item\label{leq_hf} If $B$ is HF-constructible (resp. F-constructible) from $A$, then we write $A \leq_{HF} B$ (resp. $A \leq_F B)$. Furthermore, we refer to linear orderings as in (1) as HF-orderings (resp. F-orderings) of $B$ over $A$.
	\item\label{pred_R} Given a HF-ordering of $B$ over $A$ we define a directed graph structure (cf. Definition \ref{def_directed_graph}) $(B, R_<)$ on $B$ by letting $R_<(a, b) = R(a, b)$ if $b \in B - A$, $b$ is incident with $a$ and either $a \in A$ or $a \in B-A$ and $a < b$.
\end{enumerate}
\end{definition}

	\begin{definition}\label{def_compatible} Let $A, B$ and $R$ be as in Definition \ref{HF-digraph}, and $<$ an ordering of $B - A$. We say that the order $<$ is compatible with the HF-digraph $(B, R)$ over $A$ if $R(a, b)$ implies that $a \in A$ or $a < b$.
\end{definition}

	\begin{remark}\label{remark_HF-digraph} 
	\begin{enumerate}[(1)]
	\item Let $(A, R_<)$ be as in Definition~\ref{def_nwf_cons}(\ref{pred_R}). Then $(A, R_<)$ is an HF-digraph over $B$ (cf. Definition~\ref{HF-digraph}).
	\item Let $A, B$ and $R$ be as in Definition \ref{HF-digraph}. Then there is an HF-ordering $<$ of $B$ over $A$ compatible with the HF-digraph $(B, R)$ over $A$.
\end{enumerate}
\end{remark}

	\begin{remark}\label{remark_well_founded} Notice that Definition \ref{def_nwf_cons} is consistent with Definition \ref{HF_construct_def}, i.e. if $B$ is countable and $(B - A, <)$ is a well-ordering of order type $\leq \omega$, then the two definitions of HF-construction given in Definitions \ref{HF_construct_def} and \ref{def_nwf_cons} coincide. Notice also that the more general definition of HF-ordering that we introduce (i.e. Definition~\ref{def_nwf_cons}) is not present in the literature (although variants of it were already considered by the authors in \cite[Sections 5 and 6]{paolini&hyttinen}), while the more restrictive one (i.e. Definition~\ref{HF_construct_def}) is present in various references on the subject, see e.g. \cite{sieben, ellers, ditor}. On the other hand, the consideration of non-wellfounded HF-ordering is crucial for the model-theoretic treatment of the subject, and it will be the main technical tool behind all our proofs. To give an example of the naturality of this notion in our setting, notice that, as argued in more detail in Remark \ref{remark_const_ultrapower}, for every HF-ordering~$<$ of a free projective plane $A$ and ultraproduct $A^*$ of $A$, we can extend naturally the order $<$ to an HF-ordering $<_*$ of $A^*$, but, unless the ultrafilter underlying the ultraproduct $A^*$ is principal, the HF-ordering $<_*$ is non-well-founded!
\end{remark}

	\begin{notation}\label{notation_preserving} Let $B \subseteq A$ and $D \subseteq C$ be partial planes such that $A$ admits an HF-ordering $<_A$ over $B$, and $C$ admits an HF-ordering $<_C$ over $D$, and let $f: A \rightarrow B$ be an embedding of partial planes such that $f(B) = D$. We say that $f$ is an $(R_{<_A}, R_{<_C})$-preserving embedding (or simply an $R$-preserving embedding, when the orders are clear from the context) if $R_{<_A}(b, a)$ holds if and only if $R_{<_C}(f(b), f(a))$ holds (recall Definition \ref{def_nwf_cons}(\ref{pred_R})), i.e. it is an embedding of models in the expanded language $L' = (S_1, S_2, I, R)$ (cf. Notation~\ref{notation_theory}), where the predicate $S_1$ holds of points, the predicate $S_2$ holds of lines, the relation $I$ denotes the point-line incidence relation and $R$ is interpreted as $R_{<_A}$ and $R_{<_C}$, respectively.
\end{notation}

		\begin{definition}\label{closure_def} Let $A$, $B$ and $(B-A, <)$ be as in Definition \ref{def_nwf_cons}. We define an operator $cl_< = cl$ on subsets $C$ of $B$ by declaring $cl(C)$ to be the smallest set satisfying the following requirements: 
	\begin{enumerate}[(1)]
	\item $C \subseteq cl(C)$;
	\item if $c \in cl(C)$ and there is $b \in B$ such that $R_<(b, c)$, then $b \in cl(C)$.
\end{enumerate}
Notice that from Definition \ref{def_nwf_cons}(\ref{pred_R}) it follows that $C \subseteq A$ implies $C = cl_<(C)$.
\end{definition}

	\begin{proposition} Let $A$, $B$ and $(B-A, <)$ be as in Definition \ref{def_nwf_cons}. The operator $cl = cl_<$ from Definition \ref{closure_def} is a closure operator, i.e. for $C, D \subseteq B$ we have that:
	\begin{enumerate}[(1)]
	\item $C \subseteq cl(C)$;
	\item $C \subseteq D$ implies $cl(C) \subseteq cl(D)$;
	\item $cl(cl(C)) = cl(C)$.
	\end{enumerate}
\end{proposition}


	\begin{proposition}\label{lemma_initial_seg} Let $A$, $B$ and $(B-A, <)$ be as in Definition \ref{def_nwf_cons}, and let $C \subseteq B$ be such that $cl_<(C) = C$ (cf. Definition \ref{closure_def}). Then there is an HF-ordering $<_+$ of $B$ over $C \cup A$. Furthermore, the order $<_*$ obtained concatenating $< \restriction (C - A)$ and $<_+$ is such that $(B, R_<)$ and $(B, R_{<_*})$ are isomorphic as digraphs.
\end{proposition}

	\begin{proposition}\label{prop_embed_th} Let $C$ be an open projective plane and $<$ an HF-ordering of $C$ over a projective subplane $A$ of $C$ (in this claim we allow the possibly that $A$ is degenerate, and in particular it can be $\emptyset$). Let $A \subseteq B \subseteq C$ be such that $cl_{<}(B) = B$. Then $\langle B \rangle_C \cong F(B)\leq_{HF} C$ (cf. Definitions \ref{subplane} and \ref{free_extension}). 
\end{proposition}

\begin{convention} To make proofs and arguments more direct, we will often use Propositions~\ref{lemma_initial_seg} and \ref{prop_embed_th} freely, i.e. without referring to it explicitly.
\end{convention}

	\begin{proof} This is easy and essentially well-known (see e.g. \cite[Proposition~1.5.11]{kelly}). 
\end{proof}

	\begin{proposition}\label{countable_open_HF_con} Every open partial plane admits an HF-ordering over~$\emptyset$.
\end{proposition}

	\begin{proof} Let $P$ be an open partial plane and let $X_P = X$ be the set of all finite subconfigurations of $P$ (cf. Definition \ref{subconf}). Then, by Fact \ref{open_implies_constr}, for every $A \in X$ we can find an HF-ordering $<_A$ of $A$ over $\emptyset$. Let now $\mathfrak{U}$ be an ultrafilter on $X$ such that for all $A \in X$ we have that $X_A = \{ B \in X : A \subseteq B \} \in \mathfrak{U}$ (notice that the collection of sets of the form $X_A$ have the finite intersection property, and so such an ultrafilter $\mathfrak{U}$ does exist). Now, for $A \in X$, let $<^1_A, ..., <^{n(A)}_A$ be an injective enumeration of the HF-orderings of $A$ over $\emptyset$, and, for $0 < i \leq n(A)$, let $Y^i_A = \{ B \in X : A \subseteq B \text{ and } <_B \restriction A = <^A_i \}$. Notice that $X_A = Y^1_A \cup \cdots \cup Y^{n(A)}_A$ and that for $0 < i < j \leq n(A)$ we have that $Y^i_A \cap Y^j_A = \emptyset$. Hence, being $\mathfrak{U}$ an ultrafilter, we can find a unique HF-ordering $<^*_A$ of $A$ over $\emptyset$ such that:
	$$Y_A = \{ B \in X : A \subseteq B \text{ and } <_B \restriction A = <^*_A\} \in \mathfrak{U}.$$
	Notice now that for $A, B \in X$ such that $A \subseteq B$ we have that $<^*_A = <^*_B \restriction A$. In fact, since $Y_A, Y_B \in \mathfrak{U}$, we have that $Y_A \cap Y_B \neq \emptyset$. Let $C \in Y_A \cap Y_B$, then we have that $<^*_A = <_C \restriction A$ and $<^*_B = <_C \restriction B$, from which it follows $<^*_A = <^*_B \restriction A$. Thus, we can conclude that $\bigcup_{A \in X} <^*_A$ is an HF-ordering of $P$ over $\emptyset$ (since this is an ordering and any counterexample to it being an HF-ordering is contained in an $A \in X_P$.)
\end{proof}

	
	\begin{observation}[Duality Principle for Open Projective Planes]\label{obs_duality} Let $A$ be an open projective plane and $<$ and HF-ordering of $A$ over $\emptyset$ (cf. Proposition~\ref{countable_open_HF_con}). Then the partial plane $\check{A}$ obtained switching the role of points and lines of $A$ is a projective plane and $<$ is an HF-ordering of $\check{A}$ over $\emptyset$.
\end{observation}

	\begin{proof}[Proof of Theorem~\ref{theorem_characterization}] This follows from Proposition~\ref{countable_open_HF_con}.
\end{proof}

\section{Axiomatization}\label{sec_axiomatization}

	Throughout the paper we will use the following notation:

	\begin{notation}\label{notation_theory} Throughout the rest of the paper, let $T$ be the theory of open projective planes (cf. Definition \ref{open} and recall Convention \ref{deg_conv}) in a language $L$ with two sorts $S_1$ and $S_2$ specifying the set of points and the set of lines, and a symmetric binary relation $I$ specifying the point-line incidence relation.
\end{notation}

	\begin{remark}\label{remark_const_ultrapower} Let $A \models T$, $<$ an HF-ordering of $A$ over $\emptyset$ (by Proposition~\ref{countable_open_HF_con} we can always find such an order), $R_<$ as in Definition~\ref{def_nwf_cons}(\ref{pred_R}), and $(A^*, R_<^*)$ an ultraproduct of $(A, R_<)$ (as a structure expanded with a directed edge relation), with respect to the ultrafilter $\mathfrak{U}$ on the set $I$. Then any ordering $<_*$ of $A^*$ compatible with the HF-digraph $R_<^*$ over $\emptyset$ (cf. Definition~\ref{def_compatible}) is an HF-ordering of $A^*$ over~$\emptyset$ such that $<_* \restriction A = <$. Notice that for every infinite cardinal~$\kappa$, we can choose $I$ and $\mathfrak{U}$ such that $(A_*, <_*)$ is $\kappa^+$-saturated (as a structure expanded with a linear order). Finally, notice that, unless $\mathfrak{U}$ is principal, the order $<_*$ is non-well-founded.
\end{remark}

	\begin{lemma}\label{cofinal_points} Let $A$ be an open projective plane and let $<$ be an HF-ordering of $A$ over $\emptyset$ (cf. Proposition \ref{countable_open_HF_con}). Then:
	\begin{enumerate}[(1)]
	\item the set of points (resp. lines) of $A$ is cofinal in the HF-ordering $<$.
	\item for every line $\ell$ (resp. point $p$) of $A$, the set of points of $A$ incident with $\ell$ (resp. of lines of $A$ incident with $p$) is cofinal in the HF-ordering $<$.
\end{enumerate}	
\end{lemma}

	\begin{proof} We only prove item (1), item (2) can be proved similarly (making some further considerations and using Observation~\ref{obs_duality}). Recall that by definition our projective planes are non-degenerate (cf. Convention \ref{deg_conv}), and thus we can find the following configuration $C$ in $A$:
	\begin{enumerate}[(1)]
	\item points: $a, b, c, d, e, f, g$;
	\item collineations: $ade, bce, acf, bdf, cdg, abg$.
\end{enumerate}
Here, among the points and lines of $C$, one element in $\{ e, f, g \}$ has to be the $<$-largest (since $<$ is an HF-ordering), and $e, f, g$ cannot be collinear (since $A$ is open). Suppose that $e$ is the $<$-largest. Clearly in $A$ we can find a at least one line $\ell_0 > e$ and at least one point $p_0 > \ell_0$. Also, easy inspection shows that there can not be a line of $A$ that contains $> 3$ points from $X = \{ a, b, c, d, e, f, g \}$. Hence, since $|X| = 7$ (and $3 + 3 = 6$), whenever there is a point $p > e$, there is also a line $\ell > p$. Furthermore, using Observation~\ref{obs_duality} we see that whenever there is a line $\ell > e$, there is also a point $p > \ell$. Hence, the cofinality claim follows.
\end{proof}

\begin{context}\label{notation_ultraprod} In what follows we will often work under the following assumptions, which we fix for later reference: $D$ is a model of $T$, $<$ is an HF-ordering of $D$ over~$\emptyset$, $\kappa \geq max\{ 2^{\aleph_0}, |D|^+\}$ and $\mathfrak{U}$ is an non-principal ultrafilter on some set $I$ such that the corresponding ultrapower $D^*$ is $\kappa^+$-saturated, and also the structures $(D^*, R^*_{<})$ (cf. Remark~\ref{remark_const_ultrapower}) and $(D^*, <_*)$ (again cf. Remark~\ref{remark_const_ultrapower}) are $\kappa^+$-saturated, and $<_*$ is compatible with $R^*_{<}$ (as in Remark~\ref{remark_const_ultrapower}), that is $R_{<_*} = R^*_<$. Notice that, by Lemma~\ref{cofinal_points}, the order $<_*$ has cofinality $\geq \omega_1$. For ease of notation, in what follows we will denote the order $<_*$ of $A^*$ just described simply as $<$.
\end{context}

	\begin{notation}\label{notation_embed_th} Let $A$ be a partial plane and let $X$ be the partial plane obtained from $A$ adding $\omega$-many new points $(x_i : i < \omega)$ not incident with any line of $A$. Then $F(X)$ admits a natural HF-ordering $<_+$ over $A$: the elements $(x_i : i < \omega)$ form an initial segment of $<_+$ and the rest of the order is the natural F-ordering of $F(X)$ over $X$ that we get from the definition of $F(X)$ (cf. Definition~\ref{free_extension}).
\end{notation}

	\begin{theorem}\label{embedding_th} Let $(D^*, <)$ be as in Context~\ref{notation_ultraprod} and $A \subseteq D^*$ a countable projective subplane of $D^*$ (cf. Definition~\ref{subplane}) such that $cl_<(A) = A$ (cf. Definition~\ref{closure_def}). Let $X$ be the partial plane obtained from $A$ adding $\omega$-many new points $(x_i : i < \omega)$ not incident with any line of $A$, and let $B = F(X)$ and $<_+$ the HF-ordering of $F(X)$ over $A$ described in Notation~\ref{notation_embed_th}. Then there is an $(R_{<_+}, R_{<})$-preserving embedding (cf. Notation \ref{notation_preserving}) $f: B \rightarrow D^*$ such that $f \restriction A = id_A$. Furthermore, we can choose $f$ such that every $a \in f(B - A)$ is $<$-bigger than \mbox{any given element of~$D^*$.}
\end{theorem}

	\begin{proof} Let $A$, $B$, $D^*$ and $<$ be as in the assumption of the theorem. Let $\ell$ be a line of $A$, and $a_1, a_2, a_3$ three distinct points from $D^* - A$ which are non-collinear and not incident with $\ell$. Now, as observed in Proposition \ref{cofinal_points}, the set of points of $D^*$ incident with any given line form a $<$-cofinal sequence and furthermore the order $<$ has cofinality $\geq \omega_1$, as observed in Context~\ref{notation_ultraprod}. Thus, by induction on $i < \omega$, we can find $(B_i : i < \omega)$ and $(b_i : i < \omega)$ such that:
	\begin{enumerate}[(i)]
	\item\label{first} $A \cup \{ a_1, a_2, a_3 \} \subseteq B_0$;
	\item $B_i$ is a countable projective subplane of $D^*$ and $cl_<(B_i) = B_i$;
	\item $b_i$ is incident with $\ell$ and it is $<$-bigger than any element of $B_i$;
	\item\label{last} $B_i \cup \{ b_i \} \subseteq B_{i+1}$ and $B_{i+1} = \langle B_i \cup  cl_<(b_i) \rangle_{D^*}$;
\end{enumerate}
Now, let $R = R_{<}$. Clearly, $R(\ell, b_i)$ holds, for every $i < \omega$. Furthermore, there is at most one $j \in \{ 1, 2, 3 \}$ such that $R(a_j \vee b_i, b_i)$. Thus, by pigeon hole principle, we can assume that $\neg R(a_j \vee b_i, b_i)$ holds for all $j \in \{ 1, 2 \}$ and $i < \omega$. Hence, $b_i < a_1 \vee b_i, a_2 \vee b_i$, since otherwise $R(a_1 \vee b_i, b_i)$ or $R(a_2 \vee b_i, b_i)$ holds. Let $c_i = (a_1 \vee b_{2i}) \wedge (a_2 \vee b_{2i+1})$. Now, $a_1 \vee b_{2i}, a_2 \vee b_{2i+1} < c_i$, since $b_{2i+1} < a_2 \vee b_{2i+1} = a_2 \vee c_i$, and $<$ is an HF-ordering. Also, for $i < j$, the lines $\ell_{i, j}: = c_i \vee c_j$ are such that $c_i, c_j < \ell_{i, j}$. Similarly, for all $d \in A$ we have that $d, c_i < d \vee c_i$. Now, using what we have just observed and the inductive properties (\ref{first})-(\ref{last}) listed above, it is easy to see that the canonical extension to $B = F(X)$ of the map:
	$$ a \mapsto a \; (a \in A); \;\;\; x_i \mapsto c_i \; (i < \omega)$$
 is as wanted. The ``furthermore part'' of the theorem is clear from the proof.
\end{proof}
	
	\begin{remark}\label{F_over_empty} Let $(D^*, <)$ be as in Context~\ref{notation_ultraprod} and $A \subseteq D^*$ a countable projective subplane of $D^*$ such that $cl_<(A) = A$. Then, applying Theorem \ref{embedding_th} we can find $(x_i : i < \omega)$ as there. Now, for every $1 \leq n < \omega$, we can find a point $c(n)$ in $F(X)$, such that $c(n)$ is F-constructible from $(x_i : i < 2^{2n})$, and, for every $i < 2^{2n}$, $d_{R_<}(x_i, c(n)) = 2n$ in $(F(X), R_<)$ (cf. Definitions \ref{def_path} and \ref{def_nwf_cons}(\ref{pred_R})) -- to picture the case $n = 2$ see Example~\ref{example_nis4}. Thus, by compactness and $\kappa^+$-saturation of $(D^*, <)$ and $(D^*, R^*_{<})$ (cf. Context~\ref{notation_ultraprod}), we can find a point $c_{\omega}$ in $D^*$ such that:
\begin{enumerate}[(1)]
\item the order $<$ witnesses that the point $c_{\omega}$ is F-constructible from $\emptyset$ in $D^*$;
\item every element of $cl_<(a_\omega)$ is $<$-greater than any element of $A$;
\item\label{item4} the isomorphism type of $(cl_<(a_\omega), R_<)$ is fixed and as in Construction~\ref{construction_HF_empty}.
\end{enumerate}
\end{remark}	

	\begin{example}\label{example_nis4} This example is to picture the case $n = 2$ in Remark~\ref{F_over_empty}. Let $X_8 = (x_i : i < 8)$ be the partial plane consisting of $8$ non-collinear points. We define the following F-construction over $X_8$:
	\begin{enumerate}[(1)]
	\item add a new line  $x_0 \vee x_1$;
	\item add a new line  $x_2 \vee x_3$;
	\item ...
	\item add a new line  $x_{14} \vee x_{15}$;
	\item add a new point $y_1 := (x_0 \vee x_1) \wedge (x_2 \vee x_3)$;
	\item add a new point $y_2 := (x_4 \vee x_5) \wedge (x_6 \vee x_7)$;
	\item add a new point $y_3 := (x_8 \vee x_9) \wedge (x_{10} \vee x_{11})$;
	\item add a new point $y_4 := (x_{12} \vee x_{13}) \wedge (x_{14} \vee x_{15})$;
	\item add a new line  $y_1 \vee y_2$;
	\item add a new line  $y_3 \vee y_4$;
	\item add a new point $c(2) := (y_1 \vee y_2) \wedge (y_3 \vee y_4)$.
	\end{enumerate}
\end{example}

	\begin{construction}\label{construction_HF_empty} We explain Remark~\ref{F_over_empty}(\ref{item4}), i.e. we describe the isomorphism type of $(cl_<(a_\omega), R_<)$. For all $\eta \in 2^{<\omega}$ there are distinct elements $z_\eta$ such that:
	\begin{enumerate}[(1)]
	\item $z_{\emptyset} = c_{\omega}$;
	\item if $dom(\eta)$ is even, then $z_\eta$ is a point; 
	\item if $dom(\eta)$ is odd, then $z_\eta$ is a line;
	\item if $dom(\eta)$ is even, then $z_\eta = z_{\eta \frown 0} \wedge z_{\eta \frown 1}$; 
	\item if $dom(\eta)$ is odd, then $z_\eta = z_{\eta \frown 0} \vee z_{\eta \frown 1}$;
	\item $R_<(z_{\eta \frown i}, z_\eta)$, for $i = 0, 1$.
	\end{enumerate}
\end{construction}

	\begin{lemma}\label{lemma_Fcons} Let $(D^*, <)$ be as in Context~\ref{notation_ultraprod} and $A \subseteq D^*$ a countable projective subplane of $D^*$ such that $cl_<(A) = A$. Let also $a \in D^* - A$. Then there is a countable $B \subseteq D^*$ such that:
	\begin{enumerate}[(1)]
	\item $A \cup \{ a\} \subseteq B$;
	\item $cl_{<(B)} = B$;
	\item $B$ admits an F-order $<_+$ over $A$;
	\item $B = A \cup cl_{<_+}(a)$.
\end{enumerate} 
\end{lemma}

	\begin{proof} Let $B_0 = A \cup cl_<(a)$, then clearly $cl_<(B_0) = B_0$. Let $<_0$ be $< \restriction B_0$. Suppose that $<_0$ is not an F-ordering over $A$, then there is $b \in cl_{<_0}(a) - A$ such that $k_b := |\{x \in B : R_{<_0}(x, b)\}| < 2$ (notice that $k_b \leq 2$ since $cl_<(B_0) = B_0$). We show how to deal with the case $k_b = 1$, the case $k_b = 0$ is similar. Let $x_0$ be a witness for $k_b = 1$. Without loss of generality $b$ is a point and $x_0$ is a line. Arguing as in Remark \ref{F_over_empty}, by Theorem \ref{embedding_th}, compactness and $\kappa^+$-saturation of $(D^*, <)$ and $(D^*, R^*_{<})$ (cf. Context~\ref{notation_ultraprod}), we can find $c_0, c_1, c_2$ such that they are F-constructible over $\emptyset$ by the order induced by $<$, and $cl_<(a_i) \cap (cl_<(a_j) \cup B_0) = \emptyset$ for all $i < j < 3$.
Let now $c_3 = (b \vee c_0) \wedge (c_1 \vee c_2)$, $B'_1 = B_0 \cup \bigcup_{i < 3} cl_<(c_i) \cup \{ c_3 \}$ and $B_1 = B'_1 \cup \{ p \vee q : p \neq q \text{ points of } B'_1 \}$. Then $cl_<(B_1) = B_1$ and $B_1$ admits an HF-ordering $<_1$ over $A$ in such a way that we first construct $(A \cup cl_<(b)) - \{ b \}$, then $cl_<(c_i)$, $i < 3$, then $c_3$ and then $b$. Notice that with respect to $<_1$ we have that $b$ is such that $|\{x \in B_1 : R_{<_1}(x, b)\}| = 2$, as witnessed by $x_0$ and $c_1 \vee c_2$. Nonetheless, $|\{x \in B_1 : R_{<_1}(x, c_3)\}| < 2$, but on one hand the distance (cf. Definition \ref{def_path}) between $a$ and $c_3$ in the directed graph $(B_1, R_{<_1})$ is strictly greater than the distance (cf. Definition \ref{def_path}) between $a$ and $b$ in the directed graph $(B_0, R_{<_0})$, and on the other hand letting $B_2 = B_1 \cup cl_{<_1}(a)$ we are in the same situation as before, by the choice of the $c_i$. Thus, iterating this process $\omega$-many times we find $B\subseteq D^*$ such that $cl_<(B) = B$ and $B$ is F-constructible over $A$, say by the ordering $<_{\infty} = <_{+}$, and easily we see that we have $B = A \cup cl_{<_+}(a)$, as wanted.
\end{proof}

	\begin{lemma}\label{embed_F_closed} Let $A, B \models T$ and $<_A$ and $<_B$ be HF-orderings over $\emptyset$ of $A$ and $B$, respectively. Let $C \subseteq A$ and $D \subseteq B$ be such that $cl_{<_A}(C) = C$, $cl_{<_B}(D) = D$ and $g: C \cong D$. Let $C \subseteq C^+ \subseteq A$ be such that there is an F-ordering $<_+$ of $C^+$ over $C$ and $f : C^+ \rightarrow B$ is such that $g \subseteq f$ and $f$ is an $(R_{<_+}, R_{<_B})$-preserving embedding (cf. Notation~\ref{notation_preserving}). Then $D^+ := f(C^+)$ is such that $cl_{<_B}(D^+) = D^+$.
\end{lemma}

	\begin{proof} It suffices to show that for $a \in D^+ - D$ and $b \in B$ such that $R_{<_B}(b, a)$ holds, we have that $b \in D^+$. Let $x \in C^+ - C$ be such that $f(x) = a$. Then there are $y \neq z \in D^+$ such that $R_{<_+}(y, x)$ and $R_{<_+}(z, x)$ hold. Then $R_{<_B}(f(y), a)$ and $R_{<_B}(f(z), a)$ also hold and $f(y) \neq f(z)$ (since $f$ is $R$-preserving and injective). Since $|\{c \in B : R_{<_B}(c, a) \}| \leq 2$, necessarily $b \in \{ f(y), f(z) \}$, and so $b \in D^+$.
\end{proof}

	\begin{theorem1}\label{completeness_theorem} The theory $T$ of open projective planes is complete.
\end{theorem1}

	\begin{proof} Let $M, N \models T$ and $A$ and $B$ be countable elementary substructures of $M$ and $N$, respectively. Let $(A^*, <_A)$ and $(B^*, <_B)$ be as in Context \ref{notation_ultraprod}, with respect to $A$ and $B$, respectively, and with respect to the same $\kappa$, $I$ and $\mathfrak{U}$ (cf. Context~\ref{notation_ultraprod}). We show that $A^*$ and $B^*$ are elementary equivalent, clearly this suffices. Specifically, we show that Player II has a winning strategy in the Ehrenfeucht-Fra\"iss\'e game $EF_\omega(A^*, B^*)$ of length $\omega$. We play the game as follows: after every move $n$ we have a partial isomorphism $f_n : A_n \rightarrow B_n$ such that $A_n$ and $B_n$ are countable projective subplanes of $A^*$ and $B^*$, respectively, and $cl_{<_A}(A_n) = A_n$ and $cl_{<_B}(B_n) = B_n$. For simplicity, the game starts with $n = -1$ and we let $f_{-1} = \emptyset$. Furthermore, by Observation~\ref{obs_duality}, we can assume that Player I chooses only points. Suppose then that Player~I chooses a point $a \in A^*$. By Lemma~\ref{lemma_Fcons}, we can find a countable $A_n \cup \{ a \} \subseteq C_{n+1} \subseteq A^*$ such that $cl_{<_A}(C_{n+1}) = C_{n+1}$ and $C_{n+1}$ admits an F-ordering $<_{n+1}$ of $C_{n+1}$ over $A_n$ such that $C_{n+1} = A_n \cup cl_{<_{n+1}}(a)$. We make the following claim, which we will prove below (after the end of the current proof).

 \begin{claim}\label{embed_claim} {\em There is an $(R_{<_{n+1}}, R_{<_B})$-preserving embedding (cf. Notation \ref{notation_preserving}) $g_{n+1}: C_{n+1} \rightarrow B^*$ such that $g_{n+1}$ extends $f_n$.}
\end{claim} 
\newline Now, by the claim and Lemma \ref{embed_F_closed}, we have that $D_{n+1} := g_{n+1}(C_{n+1})$ is such that $cl_{<_B}(D_{n+1}) = D_{n+1}$. It is now easy to extend $g_{n+1}$ to $f_{n+1}$ so that the domain of $f_{n+1}$ is $A_{n+1} := \langle C_{n+1} \rangle_{A^*}$ (cf. Definition~\ref{subplane}) and its codomain is $B_{n+1} := \langle D_{n+1} \rangle_{D^*}$, and clearly $cl_{<_A}(A_{n+1}) = A_{n+1}$ and $cl_{<_B}(B_{n+1}) = B_{n+1}$. Hence, player II has a winning strategy in the Ehrenfeucht-Fra\"iss\'e game $EF_\omega(A^*, B^*)$.
\end{proof}

	\begin{proof}[Proof of Claim \ref{completeness_theorem}.1] Since $C_{n+1} = A_n \cup cl_{<_{n+1}}(a)$ we have that $C_{n+1} = A_n \cup\bigcup_{0 < i < \omega} Y_i$, where, for $0 < i < \omega$, $Y_i$ is the set of points $y$ in $cl_{<_{n+1}}(a)$ such that in the directed graph $(cl_{<_{n+1}}(a), R_{<_{n+1}})$ the distance (cf. Definition \ref{def_path}) between $a$ and $y$ is $\leq 2i$, together with all the lines $\ell$ which are incident with at least two points from $Y_i \cup A_n$ such that at least one of these, say $b$, is such that $R(b, \ell)$ holds and $b \in Y_i$. For $0 < i < \omega$, let also $X_i$ be the set of all points $x \in Y_i - A_n$ such that the distance between $a$ and $x$ is exactly $2i$. Let $0 < i < \omega$ and let $D$ be the set of $x \in X_i$ such that $x$ is incident with at least one line from $A_n$ (and thus exactly one, since $A_n$ is assumed to be a projective subplane of $A^*$). For each $x \in D$ choose two distinct points $x_0, x_1 \in A^*$ such that $x_0$ is not incident with any line from $A_n \cup X_i$ and $x_1$ is incident with only $x_0 \vee x_1$. Let now $E = (X_i - D) \cup \{ x_0, x_1 : x \in D \}$. Then $A_n \cup Y_i$ is F-constructible from $A_n \cup E \cup X_i$ and $E \cup X_i$ satisfy the assumption of Theorem \ref{embedding_th}, and so we can find an $(R_{<_{n+1}}, R_{<_B})$-preserving embedding $g_{(n+1, i)}: A_n \cup Y_i \rightarrow B^*$. By compactness and $\kappa^+$-saturation of $(B^*, <_B)$ and $(B^*, R_{<_B})$ (cf. Context~\ref{notation_ultraprod}) this suffices to find the wanted $(R_{<_{n+1}}, R_{<_B})$-preserving embedding $g_{n+1}: C_{n+1} \rightarrow B^*$.
\end{proof}

\section{Elementary Substructures}\label{sec_elem_sbs}

	Recall that $T$ denotes the theory of open projective planes (cf.~Notation~\ref{notation_theory}). Also, we denote by $\preccurlyeq$ the elementary submodel relation.
	
	\begin{remark}\label{remark_induced_HF} Let $A$, $B$ and $C$ be partial planes. Then $A \leq_{HF} B$ and $C \subseteq B$ implies that $C \cap A \leq_{HF} C$ (for the definition of $\leq_{HF}$ cf. Definition \ref{def_nwf_cons}(\ref{leq_hf})).
\end{remark}

%

	\begin{lemma}\label{first_lemma_elem_sub} If $A, B \models T$, $A \subseteq B$ and $B$ is HF-constructible over $A$, then $A \preccurlyeq B$.
\end{lemma}

	\begin{proof} Let $A, B$ be as in the assumption of the lemma. Let $(A^*, <_A)$ and $(B^*, <_B)$ be as in Context \ref{notation_ultraprod}, with respect to $A$ and $B$, respectively, and with respect to the same $\kappa$, $I$ and $\mathfrak{U}$ (cf. Context~\ref{notation_ultraprod}). It is easy to see that $A^*$ embeds naturally in $B^*$, and that with respect to this embedding $B^*$ is HF-constructible over $A^*$, since $B$ is HF-constructible over $A$. Hence, without loss of generality we can assume that $A^*$ is a substructure of $B^*$ and that $B^*$ is HF-constructible over $A^*$. It suffices to show that Player II has a winning strategy in the Ehrenfeucht-Fra\"iss\'e game $EF_\omega((A^*, D), (B^*, D))$ of length $\omega$, for every finite set $D \subseteq A^*$, since this implies easily that $A \preccurlyeq B$. Notice now that for every finite set $D \subseteq A^*$ we can find $D \subseteq C \subseteq A^*$ such that $C$ is a countable projective subplane of $A^*$ and $cl_{<_A}(C) = C$. Thus, Player II wins the Ehrenfeucht-Fra\"iss\'e game $EF_\omega((A^*, D), (B^*, D))$ playing as in the proof of Theorem \ref{completeness_theorem} starting from $f_{-1} = id_D$.
\end{proof}

	\begin{notation}\label{notation_acl} Given a structure $M$ and $A \subseteq M$, we denote by $acl_M(A) = acl(A)$ the algebraic closure (in the usual first-order sense) of $A$ in $M$.
\end{notation}

	\begin{lemma}\label{lemma_acl} Suppose that $A \models T$ and $B \subseteq A$ is such that $acl_A(B) = B$ (cf. Notation \ref{notation_acl}). Then $A$ is HF-constructible over $B$.
\end{lemma}

	\begin{proof} Let $A$ and $B$ be as in the assumption of the lemma, $\kappa = (|A| + \omega)^+$, and let $D$ be such that $A \preccurlyeq D$ and $D$ is $\kappa$-saturated. Choose $A_i \preccurlyeq D$ such that:
	\begin{enumerate}[(i)]
	\item $A_0 = A$;
	\item for all $i < j < \kappa$, $A_i \cap A_j = B$;
	\item for all $i < \kappa$, there is an isomorphism $f_i: A \rightarrow A_i$ such that $f \restriction B = id_B$.
	\end{enumerate}
	By Proposition \ref{countable_open_HF_con} we have that $D$ admits an $HF$-ordering $<$ over $\emptyset$. Let $C \subseteq D$ such that $cl_<(C) = C$, $B \subseteq C$ and $|C| < \kappa$. Then there is $i < \kappa$ such that $C \cap A_i = B$. Then, since $C \leq_{HF} D$ (by Proposition \ref{lemma_initial_seg}) and $A_i \subseteq D$, by Remark~\ref{remark_induced_HF} we have that that $A_i \cap C = B \leq_{HF} A_i$, and so:
	$$ f_i^{-1}(B) = B \leq_{HF} A = f^{-1}(A_i),$$
since $f_i$ is an isomorphism and $f \restriction B = id_B$.
\end{proof}

	\begin{corollary}\label{cor_elem_sbs} If $A, B \models T$ and $A \preccurlyeq B$, then $B$ is HF-constructible over $A$.
\end{corollary}

	\begin{proof} If $A \preccurlyeq B$, then $acl_B(A) = A$, and so the claim follows from Lemma \ref{lemma_acl}.
\end{proof}

	In reading the following proof recall notation $A \leq^*_{HF} B$ from Definition~\ref{HF_construct_def}.

	\begin{proof}[Proof of Theorem \ref{th_elem_substr}] The main claim of the theorem follows from Lemma \ref{first_lemma_elem_sub} and Corollary \ref{cor_elem_sbs}. Concerning the ``furthermore part'', let $A \preccurlyeq B$ be free projective planes and suppose that $A$ is finitely generated. Notice that we can assume that also $B$ is finitely generated, since we can find a finitely generated $B' \preccurlyeq B$ such that $A \preccurlyeq B$. Let now $B^* \subseteq B$ and $A^* \subseteq A$ be finite generating configurations of the same rank as $A$ and $B$, respectively (cf. Fact~\ref{sieben_fact}). Let $B^* = B_0$ and, by induction on $n < \omega$, let $B_n$ be as in Definition~\ref{free_extension}, so that $B = \bigcup_{n < \omega} B_n$. Let $<$ be an HF-ordering of $B$ over $\emptyset$ such that for every $i < j < \omega$ we have that for every $b \in B_i$ and $b' \in B_j$ we have that $b < b'$. Let now $k < \omega$ be such that $A_* \subseteq B_k$. Now, $B_k \leq^*_{HF} B$ and $cl_<(A - B_k) = A - B_k$, and so using Proposition~\ref{lemma_initial_seg} it is easy to see that $A \cup B_k \leq^*_{HF} B$. Thus, it remains to show that $A \leq^*_{HF} A \cup (B_k - A)$, but this is clear since $B_k$ is finite and $A \leq_{HF} B$ (recall Remark~\ref{remark_induced_HF}).
\end{proof}	

\section{Prime Models}\label{sec_prime_models}

	\begin{fact}[{\cite[Theorem~3]{sandler_coll}}]\label{sandler_fact} Let $4 \leq n < \omega$ and $\pi^n = F(\pi^n_0)$ (recall Definition~\ref{pi_n}). Let $C$ be a subconfiguration of $\pi^n$ isomorphic to $\pi^n_0$ which generates $\pi^n$ (cf. Definition~\ref{subconf}(2)). Then there exists $\alpha \in Aut(\pi^n)$ such that $\alpha(\pi^n_0) = C$. In fact, enumerating $\pi^n_0$ as $(a_1, ..., a_{n+1})$ in such a way that $a_{n+1}$ is a line, $a_1, ..., a_{n-2}$ are points incident with $a_{n+1}$, and $a_{n-1}, a_n$ are points off of $a_{n+1}$, and enumerating $C$ in an analogous manner as $(b_1, ..., b_{n+1})$, $\alpha \in Aut(\pi^n)$ \mbox{can be taken so that $\alpha(a_i) = b_i$.}
\end{fact}

	\begin{lemma}\label{crucial_lemma_prime} Let $4 \leq n < \omega$ and $f: \pi^n \rightarrow \pi^n$ be a non-surjective embedding, then $f(\pi^n)$ is not an elementary substructure of $\pi^n$.
\end{lemma}

	\begin{proof} Let $A := \pi^n$ and $A' :=  f(\pi^n)$. Let $a_1, ..., a_{n+1}$ be a subconfiguration of $A$ isomorphic to $\pi^n_0$, and such that it generates $A$ (in the sense of Definition~\ref{subplane}), and such that e.g. $a_{n+1}$ is a line, $a_1, ..., a_{n-2}$ are points incident with $a_{n+1}$, and $a_{n-1}, a_n$ are points off of $a_{n+1}$. Let also $a'_i := f(a_i)$, for $1 \leq i \leq n+1$. For the sake of contradiction suppose that $A'$ is elementary in $A$. Now, since $A' \subseteq A$ and $\{ a_i : 1 \leq i \leq n+1 \}$ generates $A$, for every $1 \leq i \leq n+1$ we can find a term $t_i$ in the language $L' = \{ \wedge, \vee \}$ such that $A \models a'_i = t_i(a_1, ..., a_{n+1})$. Let $\psi(x_1, ..., x_{n+1})$ be the formula in the language of projective planes expressing that $x_{n+1}$ is a line, $x_1, ..., x_{n-2}$ are points incident with $x_{n+1}$, and $x_{n-1}, x_n$ are points off of $x_{n+1}$ (so that that $x_1, ..., x_{n+1}$ is isomorphic to $\pi^n_0$). Then we have:
$$A \models \exists x_1, ..., x_{n+1}(\psi(x_1, ...,x_{n+1}) \wedge \bigwedge^{n+1}_{i = 1} a'_i = t_i(x_1, ..., x_{n+1})).$$
Thus, since $A'$ is elementary in $A$, we have:
$$A' \models \exists x_1, ..., x_{n+1}(\psi(x_1, ...,x_{n+1}) \wedge \bigwedge^{n+1}_{i = 1} a'_i = t_i(x_1, ..., x_{n+1})).$$
But then, via the isomorphism $\alpha: A' \cong A$ such that $a'_i \mapsto a_i$, we have:
\begin{equation}\label{equstar} A \models \exists x_1, ..., x_{n+1}(\psi(x_1, ...,x_{n+1}) \wedge \bigwedge^{n+1}_{i = 1} a_i = t_i(x_1, ..., x_{n+1})). \tag{$\star$}
\end{equation}
Let $b_1, ..., b_{n+1} \in A$ be a witness of (\ref{equstar}). Then, clearly, $\{ b_1, ..., b_{n+1} \}$ is a subconfiguration of $A$ isomorphic to $\pi^n_0$. On the other hand, by the second conjunct of the formula in (\ref{equstar}), we have that:
$$\langle b_i : 1 \leq i \leq n+1 \rangle_A = A,$$
since $\{ a_i : 1 \leq i \leq n+1 \} \subseteq \langle b_i : 1 \leq i \leq n+1 \rangle_A$ and $\{ a_i : 1 \leq i \leq n+1 \}$ generates $A$. 
Hence, by Fact~\ref{sandler_fact} and the explicit definition of $\psi(x_1, ..., x_{n+1})$, we have that:
$$\beta: a_i \mapsto b_i \in Aut(A).$$
Furthermore:
$$\{ t_i(b_1, ..., b_{n+1}) : 1 \leq i \leq n+1 \} = \{a_1, ..., a_{n+1} \}$$
is a subconfiguration of $A$ isomorphic to $\pi^n_0$ which generates $A$, and so:
$$\gamma: b_i \mapsto t_i(b_1, ..., b_{n+1}) \in Aut(A).$$
Hence, we have:
$$\begin{array}{rcl}
(\beta^{-1} \circ \gamma \circ \beta)(a_i) & = & (\beta^{-1} \circ \gamma)(b_i)\\
 & = & \beta^{-1}(t_i(b_1, ..., b_{n+1}))\\
 & = & t_i(\beta^{-1}(b_1), ..., \beta^{-1}(b_{n+1}))) \\
 & = & t_i(a_1, ..., a_{n+1}) \\
 & = & a'_i. \\
\end{array}$$
So the map $\alpha^{-1}: a \mapsto a'_i = t_i(a_1, ..., a_{n+1})$ is in $Aut(A)$, a contradiction.
\end{proof}


	\begin{corollary}\label{embedding_free_planes} If $4 \leq m, n \leq \omega$ and $\pi^m$ embeds \mbox{elementarily in $\pi^n$, then $m \leq n$.}
\end{corollary}

	\begin{proof} Let $4 \leq m, n \leq \omega$ and suppose that $\pi^m$ embeds elementarily in $\pi^n$, say via the map $f$, and that $m > n$ (and so necessarily $n < \omega$). Then, by Lemma~\ref{first_lemma_elem_sub} we have that  $\pi^{n} \preccurlyeq \pi^m$ and $\pi^n \neq \pi^m$, and so $f \restriction \pi^n$ is a non-surjective elementary embedding of $\pi^n$ into itself, contradicting Lemma~\ref{crucial_lemma_prime}.
\end{proof}

	\begin{proof}[Proof of Theorem~\ref{th_prime_model}] 
If $M$ is prime, then it embeds elementarily in $\pi^4$ and so by Fact~\ref{facts_free_planes} it is isomorphic to $\pi^m$ for some $4 \leq m \leq \omega$. Hence, by Corollary~\ref{embedding_free_planes}, it suffices to show that $\pi^4$ is not prime. For the sake of contradiction, suppose that $\pi^4$ is prime. Choose copies $(A_i : i < \omega)$ of $\pi^4$ such that $A_i \subsetneq A_{i+1}$ and let $A = \bigcup_{i < \omega} A_i$. Then $A \models T$ (notice that the theory of open projective planes is a $\forall\exists$-theory). Let $f$ be an elementary embedding of $\pi^4$ into $A$, and let $B := f(\pi^4)$. Let then $i < \omega$ be such that $B \subseteq A_i$ (recall that $B$ is generated by four elements). Then, using Remark~\ref{remark_induced_HF} and Lemmas~\ref{first_lemma_elem_sub} and \ref{lemma_acl}, we have the following implications:
$$\begin{array}{rcl}
B \preccurlyeq A & \Rightarrow & B \leq_{HF} A \\
 				 & \Rightarrow & B \leq_{HF} A_{i + 1} \\
 				 & \Rightarrow & B \preccurlyeq A_{i+1}.
\end{array}$$
contradicting Lemma~\ref{crucial_lemma_prime} (letting the non-surjective embedding be $f: \pi^4 \rightarrow A_{i+1}$).
\end{proof}

	\begin{remark} Notice that already in \cite[Theorem~4.1]{kelly} it was proved that the projective $A = \bigcup_{i < \omega} A_i$ from the proof of Theorem~\ref{th_prime_model} is not a free projective plane. 
\end{remark}

\section{Algebraic and Definable Closure}\label{sec_alg_cl}


	\begin{proposition}\label{acl_neq_dcl} Let $M \models T$ be such that $\pi^4 \preccurlyeq M$. Then $acl_M \neq dcl_M$.
\end{proposition}

	\begin{proof} Suppose that $\pi^4 \preccurlyeq M$ and let $A = \pi^4$. Since $A \preccurlyeq M$, we have that $dcl_A = dcl_M$ and $acl_A = acl_M$. Hence, it suffices to show that there exists $X \subseteq A$ such that $dcl_A(X) \neq acl_A(X)$. Let $\{ a, b, c, d \}$ be a generating quadrangle of $A$ and let:
	$$e = (a \vee c) \wedge (b \vee d), \; f = (a \vee b) \wedge (c \vee d), \;  g = (a \vee d) \wedge (b \vee c).$$
	Now, the map $a \mapsto b$, $b \mapsto a$, $c \mapsto d$ and $d \mapsto c$ extends to an $\alpha \in Aut(A)$ and $\alpha$ fixes $e, f$ and $g$ and so $a, b, c, d \notin dcl_A(\{ e, f, g \})$. On the other hand, by Lemma~\ref{lemma_acl} we have that $acl_A(e, f, g) \leq_{HF} A$ and so $\{ a, b, c, d \} \in acl_A(e, f, g)$, since one cannot HF-construct the following finite configuration from any set containing $e, f$ and $g$ (given that every line contains $3$ points and every point is incident with $3$ lines):
	\begin{enumerate}[(i)]
\item points: $a, b, c, d, e, f, g$;
\item lines: $abg, cdg, adf, bcf, ace, bde$.
\end{enumerate}
\end{proof}

%

	We now generalize the definition of open from Definition~\ref{open} to open over $B$.

	\begin{definition} Let $B \subseteq A$ be finite partial planes. We say that $A$ is closed over $B$ if every element of $A - B$ is incident with at least three elements of $A$. Furthermore, given partial planes $B \subseteq A$, we say that $A$ is open over $B$ if there is no finite subconfiguration $A_0 \subseteq A$ such that $A_0$ is closed over $A_0 \cap B$.
\end{definition}

	\begin{proposition}\label{char_open_over} Let $A, B \models T$ with $B \subseteq A$. Then $A$ is open over $B$ iff $B \leq_{HF} A$. 
\end{proposition}

	\begin{proof} The obvious adaptation of the proof of Theorem~\ref{countable_open_HF_con} establishes the claim.
\end{proof}

\begin{definition}\label{free_relational_amalgam} Given structures $A, B, C$ in the same relational language with $A \subseteq B, C$ and $B \cap C = A$, we denote by $D : = B \otimes_A C$ the free amalgam of $B$ and $C$ over $A$, i.e. the domain of $D$ is $B \cup C$ and $D$ has as relations only the relations from $B$ and the relations $C$ (which is well-defined since $B \cap C = A$).
\end{definition}

\begin{definition}\label{def_canonical_amalgam} Let $A, B, C$ be open partial planes such that $B \cap C = A$, $A$ is a closed subplane of $B$ and $C$ (cf. Definition \ref{subplane}), and $A \leq_{HF} B, C$. Then we say that $D$ is the {\em canonical amalgam} of  $B$ and $C$ over $A$, denoted as $B \oplus_A C$, if $D \cong F(B \otimes_A C)$ (cf. Definitions \ref{free_extension} and \ref{free_relational_amalgam}), i.e. $D$ is the free extension (in the sense of Definition~\ref{free_extension}) of the free amalgam $B \otimes_A C$ (cf. Definition~\ref{free_relational_amalgam}).
\end{definition}

	\begin{proposition}\label{remark_can_amalgam} Referring to the context of Definition~\ref{def_canonical_amalgam}, notice that:
\begin{enumerate}[(1)]
	\item by Proposition~\ref{countable_open_HF_con} and the assumption $A \leq_{HF} B, C$ we have that: 
	$$\emptyset \leq_{HF} A \leq_{HF} B, C.$$
	\item by Proposition~\ref{char_open_over} we immediately have that $B, C \leq_{HF} D$.
	\item if $B \otimes_A C$ contains a quadrangle, then $B \oplus_A C = F(B \otimes_A C) \models T$.
	\item if $B, C \models T$ then by Lemma~\ref{first_lemma_elem_sub}  we have that $B, C \preccurlyeq B \oplus_A C$.
\end{enumerate}
\end{proposition}

	\begin{lemma}\label{lemma_partial_elementary_maps} Let $A, B \models T$, and let $A_0 \subseteq A$ and $B_0 \subseteq B$ such that $A_0, B_0 \models T$ and $A_0 \leq_{HF} A$ and $B_0 \leq_{HF} B$. Then if $f: A_0 \cong B_0$ then $f$ is partial elementary map between $A$ and $B$. In particular if $A = B = \mathfrak{M}$, where $\mathfrak{M}$ is the monster model of $T$, then there is $g \in Aut(\mathfrak{M})$ which extends $f$ (since $\mathfrak{M}$ is strongly $\kappa$-homogeneous). 
\end{lemma}

	\begin{proof} This is immediate by Proposition~\ref{remark_can_amalgam}(4).
\end{proof}

	\begin{definition}\label{pre_closure_for_acl} Let $A \models T$ and $B \subseteq A$. We define $cl^A_*(B)$ as the set of $a \in A$ such that there exists finite $A_0 \subseteq A$ such that $a \in A_0$ and $A_0$ is closed over $A_0 \cap B$. 
\end{definition}

	\begin{remark}\label{remark_A0_cup} Let $A \models T$, $B \subseteq A$ and $A_0, A'_0 \subseteq A$ finite. Notice that if $A_0$ is closed over $A_0 \cap B$ and $A_0'$ is closed over $A'_0 \cap B$, then $A_0 \cup A_0'$ is closed over $B \cap (A_0 \cup A'_0)$. 
\end{remark}

	\begin{definition}\label{closure_for_acl} Let $A \models T$ and $B \subseteq A$. By induction, we define $cl^*_n(B)$ as follows:
	\begin{enumerate}[(i)]
	\item $cl^*_{0}(B) = B$;
	\item $cl^*_{2n+1}(B) = cl^*_{2n}(B) \cup cl^A_*(cl^*_{2n+1}(B))$;
	\item (for $n> 0$) $cl^*_{2n}(B) = \langle cl^*_{2n-1}(B) \rangle_A$.
\end{enumerate}
Finally, we let $cl_A^*(B) = \bigcup_{n < \omega} cl^*_n(B)$.
\end{definition}

	\begin{remark}\label{remark_for_cl} Notice that by Proposition~\ref{char_open_over} we immediately have that $A \models T$ and $B \subseteq A$ implies that $cl_A^*(B) \leq_{HF} A$.
\end{remark}

	\begin{lemma}\label{lemma_acl_sub_cl} Let $A \models T$ and $B \subseteq A$. Then $acl_A(B) \subseteq cl_A^*(B)$.
\end{lemma}

	\begin{proof} Let $cl_A^*(B) = B_*$ and suppose that $a \notin B_*$. By induction on $n < \omega$ define:
	\begin{enumerate}[(i)]
	\item $D_0 = A$;
	\item $D_{n+1} = D_n \oplus_{B_*} A'$;
	\end{enumerate}
where $A'$ is an isomorphic copy of $A$ over $B_*$ such that $A' \cap D_n = B_*$. Notice that $D_{n+1} = D_n \oplus_{B_*} A'$ is well-defined since $B_*$ is a closed subplane of $D_n$ and $A'$ (in the sense of Definition \ref{subplane}), as it is required in Definition~\ref{def_canonical_amalgam}. Notice now that:
	\begin{enumerate}[(a)]
	\item for every $n < \omega$ we have that $D_n$ contains at least $n$ copies of $a$ over $B_*$;
	\item if $f: A \cong_{B_*} A'$ is a witness for the assertion $A'$ is an isomorphic copy of $A$ over $B_*$ such that $A' \cap D_n = B_*$, then there is an automorphism $g$ of $D_n$ which extends $f$, and so in $D_n$ we have that $tp(a/B_{*}) = tp(f(a)/B_{*})$;
	\item by Remark~\ref{remark_for_cl}, we have that $B_* \leq_{HF} A', D_n$, and thus $A \preccurlyeq D_{n}$ (cf. Proposition~\ref{remark_can_amalgam}(2)), and so $acl_{A}(B_*) = acl_{D_{n}}(B_*)$. 
\end{enumerate}
Hence, $a \notin acl_A(B_{*}) \subseteq acl_A(B)$.
\end{proof}

	\begin{lemma}\label{pre_lemma_cl_sub_acl} Let $A \models T$ and $B \subseteq A$. For all finite $A_0 \subseteq cl^A_*(B)$ (Definition~\ref{pre_closure_for_acl}) there is a finite $C \subseteq cl^A_*(B)$ such that $A_0 \subseteq C$ and such that in $A$ there are only finitely many $C' \subseteq A$ such that $C'$ is isomorphic with $C$ over $C \cap B$.
\end{lemma}

	\begin{proof} From the definition we know that for every $a \in cl^A_*(B)$ there is finite $C^a_0 \subseteq cl^A_*(B)$ such that $a \in C^a_0$ and $C^a_0$ is closed over $C^a_0 \cap B$. Let $C = \bigcup _{a \in A_0} C^a_0$. We claim that $C$ is such that in $A$ there are only finitely many copies $C'$ of $C$ over $C \cap B$. First of all, notice that by Remark~\ref{remark_A0_cup} we have that $C$ is closed over $C \cap B$. Now, if the number of copies of $C$ over $C \cap B$ is not finite, then in the monster model $\mathfrak{M}$ of $T$ the number of copies of $C$ over $C \cap B$ is unbounded, and so we can find $A \subseteq N \preccurlyeq \mathfrak{M}$ and a copy $C'$ of $C$ over $C \cap B$ such that $C' \not\subseteq N$. But then, since $C$ is closed over $C \cap B$, $C'$ is also closed over $C' \cap B$, and so $C'$ is not HF-constructible from $N$. It follows that $N \not\leq_{HF} \mathfrak{M}$, contradicting  $N \preccurlyeq \mathfrak{M}$.
\end{proof}

	\begin{lemma}\label{lemma_cl_sub_acl} Let $A \models T$ and $B \subseteq A$. For all finite $A_0 \subseteq cl^*_A(B)$ (Definition~\ref{closure_for_acl}) there is a finite $C \subseteq A$ such that $A_0 \subseteq C$ and there are only finitely many $C' \subseteq A$ such that $C'$ is isomorphic with $C$ over $C \cap B$.
\end{lemma}

	\begin{proof} It suffices to show that, for every $n < \omega$, the same happens for $cl^*_n(B)$ (cf. Definition~\ref{closure_for_acl}), and this is an easy induction where the crucial step is Lemma~\ref{pre_lemma_cl_sub_acl}.
\end{proof}

	\begin{corollary}\label{cor_char_alg_cl} Let $A \models T$ and $B \subseteq A$, then $acl_A(B) = cl_A^*(B)$.
\end{corollary}

	\begin{proof} This is immediate by Lemmas~\ref{lemma_acl_sub_cl}~and~\ref{lemma_cl_sub_acl}.
\end{proof}

	We want to notice that ideas similar to the ones occurring in the proof of Theorem~\ref{elim_quantifiers} below appear also in the proofs of related results in \cite{boney, kaisa}.
	
	\begin{proof}[Proof of Theorem~\ref{elim_quantifiers}] Let $B$ be a copy of $\pi^4$, then it is easy to find a proper substructure $A$ of $B$ which is isomorphic to $B$, and so, by Lemma~\ref{crucial_lemma_prime}, $A \models T$ is not an elementary substructure of $B \models T$. Concerning the fact that in $T$ every formula is equivalent to a Boolean combinations of $\exists$-formulas, by Lemma~\ref{lemma_partial_elementary_maps}, it suffices to show that letting $\mathfrak{M}$ be the monster model of $T$ and letting $\bar{a}, \bar{b} \in \mathfrak{M}^{< \omega}$ be such that they realize exactly the same existential formulas, then there is $g: acl_{\mathfrak{M}}(\bar{a}) \cong acl_{\mathfrak{M}}(\bar{b})$ such that $g(\bar{a}) = \bar{b}$. First of all recall that \mbox{by Corollary~\ref{cor_char_alg_cl}:}
	\begin{equation}\label{cl_equation} \tag{$\star$} acl_{\mathfrak{M}}(\bar{a}) = cl_{\mathfrak{M}}^*(\bar{a}) \; \text{ and } \; acl_{\mathfrak{M}}(\bar{b}) = cl_{\mathfrak{M}}^*(\bar{b}).
\end{equation}
Now, in order to find the wanted $g$, let $X$ be the set of all $C \subseteq acl_{\mathfrak{M}}(\bar{a})$ such that $\bar{a} \subseteq C$ and such that in $A$ there are only finitely many copies of $C$ over $\bar{a}$. Notice that by Lemma~\ref{lemma_cl_sub_acl} for all $c \in acl_{\mathfrak{M}}(\bar{a})$ there exists $C \in X$ such that $c \in C$. Now, since $\bar{a}$ and $\bar{b}$ realize exactly the same existential formulas, for every $C \in X$ there is a partial isomorphism $f_C: C \rightarrow acl_{\mathfrak{M}}(\bar{b})$ such that $f_C(\bar{a}) = \bar{b}$ (notice that the range of $f_C$ is $acl_{\mathfrak{M}}(\bar{b})$ by the choice of $X$). Let $\mathfrak{U}$ be a an ultrafiler on $\mathcal{P}(X)$ such that for all $C \in X$ we have that $X_C = \{ D \in X : C \subseteq D \} \in \mathfrak{U}$ (this can be justified by the same argument used in the proof of Theorem~\ref{countable_open_HF_con}). Now, for every $C \in X$ there are only finitely many embeddings $h$ of $C$ into $acl_{\mathfrak{M}}(\bar{b})$ such that $h(\bar{a}) = \bar{b}$ (recall the equations in (\ref{cl_equation}) above), let then $g^1_C, ..., g^{n(C)}_C$ be an injective enumeration of such embeddings and for every $0 < i \leq n(C)$ let:
	$$Y_C^i = \{ D \in X : C \subseteq D \text{ and } f_D \restriction C = g^i_C \}.$$ 
Then clearly we have that $X_C = Y_C^1 \cup \cdots \cup Y_C^{n(C)}$ and for every $0 < i < j \leq n(C)$ we have that $Y^i_C \cap Y^j_C = \emptyset$. Thus, since by the choice of $\mathfrak{U}$ we have that $X_C \in \mathfrak{U}$ and $\mathfrak{U}$ is an ultrafilter, there is a unique embedding $g_C: C \rightarrow acl_{\mathfrak{M}}(\bar{b})$ such that $g_C(\bar{a}) = \bar{b}$ and which satisfies the following condition:
$$Y_C = \{ D \in X : C \subseteq D \text{ and } f_D \restriction C = g_C \} \in \mathfrak{U}.$$ 
	Notice now that for $C, D \in X$ such that $C \subseteq D$ we have that $g_D \restriction C = g_C$. In fact, since $Y_C, Y_D \in \mathfrak{U}$, we have that $Y_C \cap Y_D \neq \emptyset$. Let $E \in Y_C \cap Y_D$, then we have that $g_C = g_E \restriction C$ and $g_D = g_E \restriction D$, from which it follows $g_C = g_D \restriction C$. Thus, we can conclude that the following map $g$ is an embedding:
	$$g = \bigcup_{C \in X} g_C: acl_{\mathfrak{M}}(\bar{a}) \rightarrow acl_{\mathfrak{M}}(\bar{b}).$$
Finally, $g$ is onto $acl_{\mathfrak{M}}(\bar{b})$ since otherwise some isomorphism type of an element of $X$ is realized over $\bar{b}$ more times than over $\bar{a}$, contradicting the fact that $\bar{a}$ and $\bar{b}$ realize exactly the same existential formulas. Thus, $g$ is the wanted isomorphism.
\end{proof}

\section{Type-Homogeneity}\label{sec_homogeneo}

\subsection{Type-Homogeneity of $\pi^n$, for $n < \omega$}\label{sec_homogeneo_omega}

	\begin{fact}[{\cite[Proposition~1.6.6]{kelly}}]\label{fact_homogeo_finite_n} Let $\pi$ be a free projective plane and let $\pi'$ be a proper non-degenerate subplane of $\pi$ such that $\pi' \leq^*_{HF} \pi$ (cf. Definition~\ref{HF_construct_def}).
Then, for any line $\ell \in \pi'$, there is a well-founded HF-ordering of $\pi$ over $\pi'$ such that in the HF-ordering there are no one point extensions of type $0$ and each one point extension of type $1$ is obtained adding a point to the line $\ell$.
\end{fact}

	\begin{lemma}\label{homogeneity_n_finite} Let $n, m < \omega$, $A = \pi^n$, $\bar{a}, \bar{b} \in A^m$ and $P \subseteq A$. Suppose that $acl_A(P\bar{a})$ is non-degenerate and that $tp(\bar{a}/P) = tp(\bar{b}/P)$. Then there exists $f \in Aut(A)$ mapping $\bar{a}$ to $\bar{b}$ and fixing $P$ pointwise.
\end{lemma}

	\begin{proof} If $tp(\bar{a}/P) = tp(\bar{b}/P)$, then there exists $\alpha: acl_A(P\bar{a}) \cong acl_A(P\bar{b})$ such that $\alpha(\bar{a}) = \bar{b}$ and $\alpha \restriction P = P$. Let $H_{\bar{a}} := acl_A(X\bar{a})$ and $H_{\bar{b}} := acl_A(X\bar{b})$. Recall that by assumption $H_{\bar{a}} \models T$ (and so also $H_{\bar{b}} \models T$). Now, by Lemma~\ref{lemma_acl} and the ``further part'' of Theorem~\ref{th_elem_substr}, $H_{\bar{a}} \leq^*_{HF} A$ and $H_{\bar{b}} \leq^*_{HF} A$ (cf. Definition~\ref{HF_construct_def}). 
	
\smallskip
\noindent By Fact~\ref{fact_homogeo_finite_n} and Proposition~\ref{lemma_initial_seg} we can find a line $\ell$ of $H_{\bar{a}}$, a well-founded HF-order $<_{\bar{a}}$ of $A$ over $H_{\bar{a}}$, and a well-founded HF-order $<_{\bar{b}}$ of $A$ over $H_{\bar{b}}$ such that letting $rk(A) - rk(H_{\bar{a}}) := k$ (notice that $k = rk(A) - rk(H_{\bar{b}})$) we have:
	\begin{enumerate}[(i)]
	\item the first $k$ elements of $<_{\bar{a}}$ are $k$ points $p_1, ..., p_k$ incident only with the line $\ell$;
	\item the first $k$ elements of $<_{\bar{b}}$ are $k$ points $q_1, ..., q_k$ incident \mbox{only with the line $\alpha(\ell)$;}
	\item $\langle H_{\bar{a}} \cup \{ p_1, ..., p_k \} \rangle_A = A = \langle H_{\bar{b}} \cup \{ q_1, ..., q_k \} \rangle_A$.
\end{enumerate}
Hence, we can define $f \in Aut(A)$ extending $\alpha$ \mbox{letting, for $i = 1, ..., k$, $f(p_i) = q_i$}.
\end{proof}

	\begin{remark} We conjecture that the non-degeneracy assumption of Lemma~\ref{homogeneity_n_finite} is not necessary but the only proof we could come up with had a complicated case distinction and so we decided to not elaborate it and to not write it down, since such proof did not add anything substantial and general to the picture.
\end{remark}

\subsection{Type-Homogeneity of $\pi^\omega$}\label{sec_homogeneo_omega}

	\begin{proof}[Proof of Corollary \ref{corollary_free}] This is by Fact \ref{facts_free_planes}(\ref{item3}) and Theorems \ref{th_complete}, \ref{th_elem_substr} and \ref{th_strictly_stable}.
\end{proof}		

\begin{context}\label{context} We refer to the general framework of \cite{baldwin_generic}. By $(\mathbf{K}, \leq)$ we denote an hereditary class $\mathbf{K}$ of finite structures in a fixed relational language $L$ together with a substructure relation $\leq$ satisfying the \mbox{following axioms (cf. \cite[Axioms Group A]{baldwin_generic}):}
\begin{enumerate}[(1)]
	\item if $A \in \mathbf{K}$, then $A \leq A$;
	\item if $A \leq B$, then $A$ is a substructure of $B$;
	\item if $A, B, C \in \mathbf{K}_0$ and $A \leq B \leq C$, then $A \leq C$;
	\item $\emptyset \in \mathbf{K}$ and $\emptyset \leq A$, for all $A \in \mathbf{K}_0$;
	\item if $A, B, C \in \mathbf{K}$, $A \leq C$, and $B$ is a substructure of $C$, then $A \cap B \leq B$.
\end{enumerate}
\end{context}

	\begin{definition}\label{various_def} Let $(\mathbf{K}, \leq)$ be as in Context \ref{context}. 
	\begin{enumerate}[(1)]
	\item We say that $B$ is a primitive extension of $A$ if $A \leq B$ and there is no $A \subsetneq B_0 \subsetneq B$ such that $A \leq B_0 \leq B$.
	\item We write $A < B$ to mean that $A \leq B$ and $A \neq B$.
\end{enumerate}
\end{definition}

\begin{notation}\label{notation_strong} Let $(\mathbf{K}, \leq)$ be as in Context \ref{context}.
	\begin{enumerate}[(1)]
	\item We denote by $\hat{\mathbf{K}}$ the class of structures $M$ in the language $L$ such that every finite substructure of $M$ is in $\mathbf{K}$.
	\item We extend $\leq$ to a relation $\leq^*$ on $\mathbf{K} \times \hat{\mathbf{K}}$ by declaring $A \leq M$ if and only if $A \subseteq M$ and for every finite $B \in \mathbf{K}$ such that $A \subseteq B \subseteq M$ we have that $A \leq B$.
	\end{enumerate}	 
\end{notation}

	\begin{definition}\label{def_semigeneric} Given $(\mathbf{K}, \leq)$ as in Context \ref{context} we say that a model $M \in \hat{\mathbf{K}}$ is $(\mathbf{K}, \leq)$-rich if:
	\begin{enumerate}[(1)]
	\item whenever $A \leq^* M$ and $A \leq B \in \mathbf{K}$ there exists $B' \leq^* M$ such that $B' \cong_A B$.
\end{enumerate}
	We say that $M \in \mathbf{K}$ is $(\mathbf{K}, \leq)$-generic if in addition:
	\begin{enumerate}[(2)]
	\item $M$ is countable and $M$ is the union of a $\leq$-chain of $\mathbf{K}$-structures.
\end{enumerate}	
\end{definition}

	\begin{fact}[\cite{baldwin_generic}]\label{fact_generic} If $(\mathbf{K}, \leq)$ satisfies the conditions of Context \ref{context} and it is an amalgamation class, then there exists a $(\mathbf{K}, \leq)$-generic $G(\mathbf{K}, \leq) = G$. Furthermore, this structure is unique up to isomorphism (with respect to being $(\mathbf{K}, \leq)$-generic) and it satisfies the following form of homogeneity: if $A, B \in \mathbf{K}$, $A, B \leq^* G$ (cf. Notation~\ref{notation_strong}(2)) and $f: A \cong B$, then there is $\tilde{f} \in Aut(G)$ such that $f \subseteq \tilde{f}$.
\end{fact}

	\begin{convention}\label{def_conc_K} From now till the end of this section $(\mathbf{K}, \leq)$ will be as below.
	\begin{enumerate}[(1)]
	\item Let $\mathbf{K}$ be the class of finite open partial planes (cf. Definition~\ref{open}) in a language $L = \{ S_1, S_2, I \}$ specifying the set of points, the set of lines and the point-line incidence relation (i.e. $L$ is as in Notation \ref{notation_theory}).
	\item For $A, B \in \mathbf{K}$, we let $A \leq B$ if $A \leq_{HF} B$, i.e. $B$ is HF-constructible from $A$  (cf. Definition~\ref{HF_construct_def}). We will stick to the notation $\leq_{HF}$ (instead of simply $\leq$) for uniformity of notation with the rest of the paper.
\end{enumerate}		
\end{convention}

\begin{lemma} (Recall Convention \ref{def_conc_K}) $(\mathbf{K}, \leq_{HF})$ satisfies Context~\ref{context}.
\end{lemma}

	\begin{proof} The only non-trivial fact is the satisfaction of conditions (4) and (5) of Context \ref{context}. The satisfaction of these conditions was already observed in the sections above (and used repeatedly), but for completeness we give full references: condition (4) is \cite[Lemma 1]{sieben} and condition (5) is \cite[Lemma 1.5.7]{kelly}.
\end{proof}

\begin{observation}\label{char_primitive} Let $B <_{HF} C \in \mathbf{K}$ be a primitive extension (cf. Definition \ref{various_def}). Then there are three cases:
	\begin{enumerate}[(1)]
	\item $C = B \cup \{ c \}$ and $c$ is incident with exactly two elements of $B$;
	\item $C = B \cup \{ c \}$ and $c$ is incident with exactly one element of $B$;
	\item $C = B \cup \{ c \}$ and $c$ is not incident with any element of $B$.
	\end{enumerate}
\end{observation}
	
	We invite the reader to recall Definition~\ref{free_relational_amalgam}.

	\begin{lemma}\label{lemma_sharp} $(\mathbf{K}, \leq_{HF})$ is an amalgamation class.
\end{lemma}

	\begin{proof} Let $A, B, C \in \mathbf{K}$, with $C <_{HF} A$ primitive and $C \leq_{HF} B$, we will show that either $A \otimes_C B \in \mathbf{K}$ or there is a $\leq_{HF}$-embedding of $A$ into $B$ over $C$. Clearly this suffices. We make a case distinction following Observation \ref{char_primitive}:
	\newline \underline{Case 1}. $A = C \cup \{ a \}$ and $a$ is not incident with any element of $C$.
	\newline In this case clearly we have that $D  =A \otimes_C B \in \mathbf{K}$ and $A, B \leq_{HF} D$.
	\newline \underline{Case 2}. $A = C \cup \{ a \}$ and $a$ is incident with exactly one element of $C$.
	\newline Also in this case clearly we have that $D  =A \otimes_C B \in \mathbf{K}$ and $A, B \leq_{HF} D$.
	\newline \underline{Case 3}. $A = C \cup \{ a \}$ and $a$ is incident with exactly two elements of $C$.
	\newline Without loss of generality $a$ is a point and $A \models a = \ell_1 \wedge \ell_2$, for $\ell_1, \ell_2$ lines of $C$. If there is no point $b \in B$ such that $B \models \ell_1 \wedge \ell_2 = b$, then clearly $D :=A \otimes_C B \in \mathbf{K}$ and $A, B \leq_{HF} D$. Thus, suppose that there is $b \in B$ such that $B \models \ell_1 \wedge \ell_2 = b$. Since $C \leq_{HF} B$, a part from $\ell_1$ and $\ell_2$ there are no other lines of $C$ which are incident with $b$, and so we can $\leq_{HF}$-embed $A$ into $B$ over $C$ mapping $a$ to $b$.
\end{proof}
	
	\begin{proposition}\label{prop_generic} $\pi^\omega$ is the $(\mathbf{K}, \leq_{HF})$-generic (cf. Definition \ref{def_semigeneric}).
\end{proposition}

	\begin{proof} Notice that the $(\mathbf{K}, \leq_{HF})$-generic $G$ is a countable projective plane well-foundedly HF-constructible over $\emptyset$, and so, by Fact~\ref{sieben_fact}, we have that $G$ is isomorphic to $\pi^n$ for some $4 \leq n \leq \omega$. Finally, by the ``furthermore part'' of Fact~\ref{sieben_fact}, we immediately get that $G$ is isomorphic to $\pi^\omega$, as wanted.
\end{proof}

	\begin{remark}\label{leq_star} Notice that if $A \in \mathbf{K}$, then $A \leq^*_{HF} \pi^\omega$ (in the sense of Notation~\ref{notation_strong}(2)) iff there is an $HF$-construction of $\pi^\omega$ over $A$ of order type $\omega$, i.e. if and only if $\pi^\omega$ is well-foundedly HF-constructible over $A$ (cf. Definition~\ref{HF_construct_def} and the coherent notation introduced there). Furthermore, by Fact~\ref{fact_generic} and Proposition~\ref{prop_generic}, if $A, B \in \mathbf{K}$, $A, B \leq^*_{HF} \pi^\omega$ and $f: A \cong B$, then there is $\tilde{f} \in Aut(\pi^\omega)$ such that $f \subseteq \tilde{f}$ (this is crucial for the proof of Theorem~\ref{th_type_homogeneity}). 
\end{remark}

	\begin{lemma}\label{lemma_generic} Let $D = \pi^n$ for some $4 \leq n \leq \omega$, $B_0 \subseteq D$ a finite set of parameters, and $a = (a_1, ..., a_n)$ and $b = (b_1, ..., b_n)$ be tuples from $D$ such that $tp(a/B_0) = tp(b/B_0)$. Then there are finite $A, B \subseteq D$ and $f: A \cong B$ such that:
	\begin{enumerate}[(1)]
	\item $aB_0 \subseteq A$ and $bB_0 \subseteq B$;
	\item $f(B_0) = id_{B_0}$;
	\item $f(a) = f(b)$;
	\item $A \leq^*_{HF} D$ and $B \leq^*_{HF} D$ (recall Remark~\ref{leq_star}).
\end{enumerate}
	\end{lemma}

	\begin{proof} Let $D \preccurlyeq M$ be saturated. Let $<_0$ be an HF-ordering of $D$ over $\emptyset$ of order type $\omega$, and let $A^+ = acl(aB_0) \subseteq D$. Then $<_1 = <_0 \restriction A^+$ is an HF-order of $A^+$ over $\emptyset$, and thus we can find finite $A \subseteq A^+$ such that $aB_0 \subseteq A$ and $A \leq_{HF} A^+$. By Lemma \ref{lemma_acl}, $A^+ \leq_{HF}D$ and, since $A \leq_{HF} A^+$, we have that $A \leq_{HF} D$. Let now $g \in Aut(M/B_0)$ be be such that $g(a) = b$ (recall that $tp(a/B_0) = tp(b/B_0)$), and let $B := g(A)$. Notice that:
	$$g(acl(aB_0)) = acl(bB_0) \subseteq D,$$
and so:
 	$$B \leq_{HF} g(acl(aB_0)) = acl(bB_0) \leq_{HF} D.$$
Let then $f = g \restriction A$. We know that $A, B \leq_{HF} D$, and so if we can show that $A, B \leq_{HF}^* D$, then $A, B$ and $f:A \cong B$ are as wanted. We prove this for $A$, the case of $B$ is analogous. To this extent, let $A^* \subseteq D$ be finite and such that $A \subseteq A^*$ and $cl_{<_0}(A^*) = A^*$. Since $A \leq_{HF} D$ and $A \subseteq A^*$, we have that $A \leq_{HF} A^*$. On the other hand, since $cl_{<_0}(A^*) = A^*$, we have that $D$ is $HF$-constructible over $A^*$ via a construction of order type $\omega$. The claim follows from the fact that $A^*$ is finite.
\end{proof}

	\begin{proof}[Proof of Theorem~\ref{th_type_homogeneity}] The claim about $\pi^\omega$ is by Proposition \ref{prop_generic}, Remark~\ref{leq_star}, and Lemma \ref{lemma_generic}. The ``furthermore part'' is by Lemma~\ref{homogeneity_n_finite}.
\end{proof} 

\section{Non-Superstability and Number of Models}\label{sec_unsuper}

	\begin{construction}\label{contruction} Let $\kappa$ be an infinite cardinal and $A$ a partial plane consisting of $\kappa$-many distinct points $(p_i : i < \kappa)$ and $\kappa$-many distinct lines $(\ell_i : i < \kappa)$ with no incidences between them. Let $p_{i_0} = p'_0, ..., p_{i_3} = p'_3$ and $\ell_{i_0} = \ell'_0, ..., \ell_{i_3} = \ell'_3$ be distinct points and lines of $A$. Consider now the following HF-construction over $A$ of an open partial plane~$B$:
	\begin{enumerate}[(1)]
	\item add a new point $a_1$ not incident with any given line;
	\item add the line  $p'_{0} \vee a_1 := \ell'_4$;
	\item add the point $b_0 := \ell'_{0} \wedge (p'_{0} \vee a_1)$;
	\item add the line  $p'_{1} \vee a_1 := \ell'_5$;
	\item add the point $b_1 := \ell'_{1} \wedge (p'_{1} \vee a_1)$;
	\item add the line  $p'_{2} \vee a_1 := \ell'_6$;
	\item add the point $b_2 := \ell'_{2} \wedge (p'_{2} \vee a_1)$;
	\item add the line  $p'_{3} \vee a_1 := \ell'_7$;
	\item add the point $b_3 := \ell'_{3} \wedge (p'_{3} \vee a_1)$;
	\item add the line  $b_0 \vee b_1 := \ell'_8$;
	\item add the line  $b_2 \vee b_3 := \ell'_9$;
	\item add the point $c_0 := (b_0 \vee b_1) \wedge (b_2 \vee b_3)$;
	\item add the line  $b_0 \vee b_2 := \ell'_{10}$;
	\item add the line  $b_1 \vee b_3 := \ell'_{11}$;
	\item add the point $c_1 := (b_0 \vee b_2) \wedge (b_1 \vee b_3)$;
	\item add the line  $b_1 \vee b_2 := \ell'_{12}$;
	\item add the line  $b_0 \vee b_3 := \ell'_{13}$;
	\item add the point $c_2 := (b_1 \vee b_2) \wedge (b_0 \vee b_3)$;
	\item add the line  $a_1 \vee c_0 := \ell'_{14}$;
	\item add the line  $c_1 \vee c_2 := \ell'_{15}$;
	\item add the point $a_{0} := (a_1 \vee c_0) \wedge (c_1 \vee c_2)$.
	\end{enumerate}
For a representation of the relevant part of $B$ see the table in Figure \ref{myfigure}, where the lines of the table represent the relevant points of $B$, the columns of the table represent the relevant lines of $B$ and the symbol ``1'' encodes the incidence relation.
\begin{figure}[h]
$$\begin{array}{|c|c|c|c|c|c|c|c|c|c|c|c|c|c|c|c|c|}
\hline
\phantom{a} & \ell'_0 & \ell'_1 & \ell'_2 & \ell'_3 & \ell'_4 & \ell'_5 & \ell'_6 & \ell'_7 & \ell'_8 & \ell'_9 & \ell'_{10} & \ell'_{11} & \ell'_{12} & \ell'_{13} & \ell'_{14} & \ell'_{15} \\
\hline
	   p'_0    &   &   &   &   & 1 &   &   &   &   &   &   &   &   &   &   &   \\
\hline p'_1    &   &   &   &   &   & 1 &   &   &   &   &   &   &   &   &   &   \\
\hline p'_2    &   &   &   &   &   &   & 1 &   &   &   &   &   &   &   &   &   \\
\hline p'_3    &   &   &   &   &   &   &   & 1 &   &   &   &   &   &   &   &   \\
\hline a_1     &   &   &   &   & 1 & 1 & 1 & 1 &   &   &   &   &   &   & 1 &   \\
\hline b_0     & 1 &   &   &   & 1 &   &   &   & 1 &   & 1 &   &   & 1 &   &   \\
\hline b_1     &   & 1 &   &   &   & 1 &   &   & 1 &   &   & 1 & 1 &   &   &   \\
\hline b_2     &   &   & 1 &   &   &   & 1 &   &   & 1 & 1 &   & 1 &   &   &   \\
\hline b_3     &   &   &   & 1 &   &   &   & 1 &   & 1 &   & 1 &   & 1 &   &   \\
\hline c_0     &   &   &   &   &   &   &   &   & 1 & 1 &   &   &   &   & 1 &   \\
\hline c_1     &   &   &   &   &   &   &   &   &   &   & 1 & 1 &   &   &   & 1 \\
\hline c_2     &   &   &   &   &   &   &   &   &   &   &   &   & 1 & 1 &   & 1 \\	
\hline a_0     &   &   &   &   &   &   &   &   &   &   &   &   &   &   & 1 & 1 \\	
\hline
\end{array}
$$\caption{A representation of $B$ from Construction \ref{contruction}.  \label{myfigure}}
\end{figure}
{\em Notice that every HF-construction of $B$ over $A$ is such that the point $a_{0}$ has to be last in the construction, since every other element of $B - A$ is incident with at least three element of $B$!} Notice also that choosing $p_{i'_0}, ..., p_{i'_3}$ and $\ell_{i'_0}, ..., \ell_{i'_3}$ distinct from the previously chosen elements we can find an HF-construction such that it starts by adding a point $a'_{2}$ not incident with any element of $A$ and ends with an element $a'_1$ in such a way that the pair $(a'_{2}, a'_1)$ over $p_{i'_0}, ..., p_{i'_3}$ and $\ell_{i'_0}, ..., \ell_{i'_3}$ looks like the pair $(a_{1}, a_{0})$ over $p_{i_0}, ..., p_{i_3}$ and $\ell_{i_0}, ..., \ell_{i_3}$. This shows that we can iterate the construction for any given finite length, i.e. we can in the same manner find HF-chains $a^n_n < a^n_{n-1} < \cdots < a^n_0$, for every $n < \omega$.
\end{construction}

\bigskip

\bigskip

	\begin{theorem}\label{unsuperstability} The theory $T$ of open projective planes is not superstable.
	\end{theorem}
	
	\begin{proof} For the sake of a contradiction, suppose that $T$ is superstable. By Fact~\ref{sieben_fact} the free projective plane $\pi^\omega$ can be considered as generated by a partial plane $A$ (i.e. $\pi^\omega = F(A)$ (cf. Definition \ref{free_extension})) consisting of $\omega$-many distinct points and $\omega$-many distinct lines with no incidences between them. Let now $D = \pi^\omega$ and consider the ultraproduct $(D^*, <)$ as in Context \ref{notation_ultraprod}, and in particular let $\kappa \geq 2^{\aleph_0}$, $I$ and $\mathfrak{U}$ be as there. Let $A^*$ consists of the $\mathfrak{U}$-equivalence ($\mathfrak{U}$ is the ultrafilter) classes of functions with values in $A$. Then in $A^*$ we can find $\kappa$-many points $P = (p_i : i < \kappa)$ and $\kappa$-many lines $L = (\ell_i : i < \kappa)$ with no incidences between them, and furthermore $cl_<(A^*) = A^*$, and so, by Proposition~\ref{lemma_initial_seg}, without loss of generality we can assume that $< \restriction A^*$ is an initial segment of $<$. Let $\omega^*$ be the set $\omega$ with the reverse ordering $<^*$. Given $X \subseteq \kappa$ we let $X_0 = \{ p_i : i \in X \} \cup \{ \ell_i : i \in X \}$. Then, iterating $\omega$-many times Construction~\ref{contruction}, and using compactness and $\kappa^+$-saturation of $(D^*, <)$ and $(D^*, R_<)$, for every countably infinite $X \subseteq \kappa$ we can find $\vec{a}_X = (a^X_i : i \in \omega^*)$, $(b^X_{(i, j)} : i \in \omega^*, j < 4)$, $(c^X_{(i, j)} : i \in \omega^*, j < 3)$ and $(\ell^X_{(i, j)} : i \in \omega^*, j < 16)$ such that letting $\hat{X}_0$ to be:
	$$X_0 \cup (b^X_{(i, j)} : i \in \omega^*, j < 4) \cup (c^X_{(i, j)} : i \in \omega^*, j < 3) \cup (\ell^X_{(i, j)} : i \in \omega^*, j < 16),$$
we have:
	\begin{enumerate}[(i)]
	\item $cl_<(\vec{a}_X \cup \hat{X}_0) = \vec{a}_X \cup \hat{X}_0$;
	\item for every $i \in \omega^* - \{ 0 \}$ there is $Z^i = \{ p^{i}_{i_0}, ..., p^{i}_{i_3} \} \cup \{ \ell^{i}_{i_0}, ..., \ell^{i}_{i_3}\} \subseteq X_0$ such that $a^X_{i-1}$ is HF-constructed from $Z^i \cup \{ a^X_{i}\}$ as in Construction \ref{contruction}, where the $b_j$ and $c_j$ there are the $b^X_{(i, j)}$ and $c^X_{(i, j)}$ here, and the $\ell'_j$ there are the $\ell^X_{(i, j)}$ here;
	\item for every $i \neq j \in \omega^* - \{ 0 \}$ we have that $Z^i \cap Z^j = \emptyset$;
	\item for every $x \in X_0$ there is $i \in \omega^* - \{ 0 \}$ such that $x \in Z^i$;
	\item if $i <^* j$, then $a^X_i$ is $<$-smaller than $a^X_j$.
\end{enumerate}	
	Notice that from the above it follows that:
	\begin{enumerate}[(vi)]
	\item any HF-construction in $D^*$ over $X_0$ containing $\vec{a}_X$ is such that for every $j <^* 0$ the element $a^X_j$ is constructed before the element $a^X_0$.
\end{enumerate}	
Now, to reach a contradiction with the assumption of superstability, it suffices to show that if $X, Y \subseteq \kappa$ are countably infinite and $X \cap Y$ is finite, then in $A^*$ we have that $tp(a_0^X/\hat{X}_0 \cup \hat{Y}_0) \neq tp(a_0^Y/\hat{X}_0 \cup \hat{Y}_0)$. By superstability of $T$, we can find $C^* \preccurlyeq D^*$ such that:
\begin{enumerate}[(i')]
	\item $C^*$ is saturated;
	\item $|C^*| = \kappa^+$;
	\item $cl_<(C^*) = C^*$;
	\item $\hat{X}_0 \cup \hat{Y}_0 \subseteq C^*$;
	\item $\vec{a}_X, \vec{a}_Y \subseteq C^*$.
\end{enumerate} 
Suppose now that $tp(a_0^X/\hat{X}_0 \cup \hat{Y}_0) = tp(a_0^Y/\hat{X}_0 \cup \hat{Y}_0)$, then there is $f \in Aut(B/\hat{X}_0 \cup \hat{Y}_0)$ such that $f(a^X_0) = f(a^Y_0)$. Since $X \cap Y$ is finite we can find $i \in \omega^* - \{ 0, 1 \}$ such that $\{ p^{i}_{i_0}, ..., p^{i}_{i_3} \} \cup \{ \ell^{i}_{i_0}, ..., \ell^{i}_{i_3} \} \subseteq X_0 - Y_0$. Notice now that $a^X_{i-1}$ is incident with $c^X_{(i, 1)} \vee c^X_{(i, 2)}$, while no element from $\vec{a}_Y$ satisfies this, and so $a^X_{i-1} \notin \vec{a}_Y$, from which it follows that $f(a^X_{i-1}) \notin \vec{a}_Y$, since $f\restriction \hat{X}_0 \cup \hat{Y}_0 = id_{\hat{X}_0 \cup \hat{Y}_0}$. Hence, $a^X_{i-1} = f^{-1}f(a^X_{i-1}) \notin f^{-1}(\vec{a}_Y)$. On the other hand, $a^X_0 = f^{-1}(a^Y_0)$ is HF-constructible from $X_0 \cup Y_0$  following the order $(f^{-1}(a^Y_i) : i \in \omega^*)$ in such a way that the construction extends to a construction of $B^*$, since $f \in Aut(B/\hat{X}_0 \cup \hat{Y}_0)$ and $cl_<(\vec{a}_Y \cup \hat{Y}_0) = \vec{a}_Y \cup \hat{Y}_0$. But this contradicts the fact that $a^X_{i-1} = f^{-1}f(a^X_{i-1}) \notin f^{-1}(\vec{a}_Y)$, since for every $j \in \omega^* - \{ 0 \}$ we have that $a^X_j$ is contained
in some finite set $E \subseteq B$ such that if $H$ is any set which contains $X_0 \cup Y_0$ and $a^X_0$ but not $a^X_j$, then all the elements of $E$ are not constructible from $H$. Hence, the assumption of superstability is contradictory.
\end{proof}


	\begin{remark}\label{remark_number_models} As observed there, in Construction \ref{contruction} we can iterate the construction of $a_0$ from $a_1$ and $p_{i_0}, ..., p_{i_3}$ and $\ell_{i_0}, ..., \ell_{i_3}$ to a construction of $a_1$ from $a_2$ and a set of points and lines $p_{i'_0}, ..., p_{i'_3}$ and $\ell_{i'_0}, ..., \ell_{i'_3}$ distinct from $p_{i_0}, ..., p_{i_3}$ and $\ell_{i_0}, ..., \ell_{i_3}$. Notice that actually we can iterate the construction also using the same points and lines, i.e. we can construct $a_1$ from $a_2$ and $p_{i_0}, ..., p_{i_3}$ and $\ell_{i_0}, ..., \ell_{i_3}$, and then construct $a_0$ from $a_1$ also using $p_{i_0}, ..., p_{i_3}$ and $\ell_{i_0}, ..., \ell_{i_3}$. Hence, if we fix two disjoint sets $Z_0 = \{ p^0_{0}, ..., p^0_{3} \} \cup \{ \ell^0_{0}, ..., \ell^0_{3}\}$ and $Z_1 = \{ p^1_{0}, ..., p^1_{3} \} \cup \{ \ell^1_{0}, ..., \ell^1_{3}\}$ we can choose at each stage $n < \omega$ if we construct $a_{n}$ from $a_{n+1}$ using $Z_0$ {\em or} $Z_1$. 
\end{remark}

	\begin{theorem}\label{cont_many_countable_open} There are continuum many countable open projective planes.
\end{theorem}

	\begin{proof} We begin the proof as in the proof of Theorem \ref{unsuperstability}, and so we refer to the objects introduced there, in particular $(D^*, <)$ and $A^*$ are as there. From $A^*$ choose disjoint sets $Z_0 = \{ p^0_{0}, ..., p^0_{3} \} \cup \{ \ell^0_{0}, ..., \ell^0_{3}\}$ and $Z_1 = \{ p^1_{0}, ..., p^1_{3} \} \cup \{ \ell^1_{0}, ..., \ell^1_{3}\}$ as in Remark~\ref{remark_number_models}. Notice that $cl_<(Z_0 \cup Z_1) = Z_0 \cup Z_1$, and so, by Proposition~\ref{lemma_initial_seg}, without loss of generality we can assume that $< \restriction A^*$ is an initial segment of $<$. Let $\omega^*$ be the set $\omega$ with the inverse ordering $<^*$. Iterating $\omega$-many times Construction~\ref{contruction} (cf. also Remark~\ref{remark_number_models}), and using compactness and $\kappa^+$-saturation of $(D^*, <)$ and $(D^*, R_<)$, for every $\eta \in 2^\omega$ we can find $\vec{a}_\eta = (a^\eta_i : i \in \omega^*)$, $(b^\eta_{(i, j)} : i \in \omega^*, j < 4)$, $(c^\eta_{(i, j)} : i \in \omega^*, j < 3)$ and $(\ell^\eta_{(i, j)} : i \in \omega^*, j < 16)$ such that letting $\hat{Z}_\eta$ to be:
	$$Z_0 \cup Z_1 \cup (b^\eta_{(i, j)} : i \in \omega^*, j < 4) \cup (c^\eta_{(i, j)} : i \in \omega^*, j < 3) \cup (\ell^\eta_{(i, j)} : i \in \omega^*, j < 16),$$
we have:
	\begin{enumerate}[(i)]
	\item $cl_<(\vec{a}_\eta \cup \hat{Z}_\eta) = \vec{a}_\eta \cup \hat{Z}_\eta$;
	\item for every $i \in \omega^* - \{ 0 \}$ we have that $a^\eta_{i-1}$ is HF-constructed from $Z_{\eta(i-1)} \cup \{ a^\eta_{i} \}$ as in Construction \ref{contruction} (cf. also Remark \ref{remark_number_models}), where the $b_j$ and $c_j$ there are the $b^\eta_{(i, j)}$ and $c^\eta_{(i, j)}$ here, and the $\ell'_j$ there are the $\ell^\eta_{(i, j)}$ here;
	\item if $i <^* j$, then $a^{\eta}_i$ is $<$-smaller than $a^{\eta}_j$;
	\item any HF-construction in $D^*$ over $Z_0 \cup Z_1$ containing $\vec{a}_\eta$ is such that for every $j <^* 0$ the element $a^{\eta}_j$ is constructed before the element $a^{\eta}_0$.
\end{enumerate}	
Now, by (i) above, we have that $B_\eta =: \vec{a}_\eta \cup \hat{Z}_\eta$ is such that $cl_<(B_{\eta}) = B_{\eta}$ and so, by Lemma \ref{first_lemma_elem_sub}, we have that the smallest projective subplane of $D^*$ containing $B_{\eta}$ is such that $A_{\eta} \preccurlyeq D^*$, and clearly $A_{\eta}$ is countable. Notice now that $B_\eta$ is the least subset of $A_\eta$ such that it contains $\{ a^{\eta}_0 \} \cup Z_0 \cup Z_1$ and from which $A_\eta$ can be HF-constructed. Thus, in $A_\eta$, from $\{ a^{\eta}_0 \} \cup Z_0 \cup Z_1$ the structure $B_\eta$ can be recognized, and from the isomorphism type of $B_\eta$ over $\{ a^{\eta}_0 \} \cup Z_0 \cup Z_1$ the function $\eta$ can be calculated.
Hence, if we let $A^*_\eta$ be what we get from $A_\eta$ by naming $\{ a^{\eta}_0 \} \cup Z_0 \cup Z_1$ by new constants, we have that $A^*_\eta \not\cong A^*_{\xi}$, if $\eta \neq \xi$. But then, since $\{ a^{\eta}_0 \} \cup Z_0 \cup Z_1$ is finite, also the set $\{ A_\eta : \eta \in 2^\omega \}$ contains $2^{\aleph_0}$-many non-isomorphic models.
\end{proof}

%

\section{Stability and Forking}\label{sec_stability}

	\begin{theorem}\label{stability} $T$ is stable.
\end{theorem}

	\begin{proof} Let $\kappa$ be infinite, $A \models T$ with $|A| = \kappa$ and $A \preccurlyeq B$ such that $B$ is $\kappa^+$-homogeneous (with respect to elementary substructures). Let also $<$ be an HF-ordering of $B$ over $\emptyset$ such that $A$ is an initial segment of $<$ (this is possible by Proposition~\ref{countable_open_HF_con} and Corollary~\ref{cor_elem_sbs}). Let now $a, b \in B$, $A_0 = cl_{<}(a)$ and $B_0 = cl_{<}(b)$. Suppose that $A_0 \cap A = B_0 \cap A$ and that there is an isomorphism $f: A_0 \cong B_0$ such that $f \restriction A_0 \cap A = id_{A_0 \cap A}$ and $f(a) = b$. Since $|A_0| \leq \aleph_0$ it suffices to show that there is $g \in Aut(B/A)$ such that $f \subseteq g$. Now, if we let $h: AA_0 \rightarrow AB_0$ be such that $h \restriction A = id_A$ and $f \subseteq h$, then $h$ is an isomorphism. It follows that $h$ extends to an isomorphism $h^*: F(AA_0) \cong  F(AB_0)$ (cf. Definition~\ref{free_extension}). Furthermore, since $cl_<(AA_0) = AA_0$ and $cl_<(AB_0) = AB_0$, by Proposition~\ref{prop_embed_th}:
	$$ F(AA_0) \cong_{AA_0} \langle AA_0 \rangle_B \models T \;\; \text{ and } \;\; F(AB_0) \cong_{AB_0} \langle AB_0 \rangle_B \models T.$$
Also, again by Proposition~\ref{prop_embed_th}, we have:
$$\langle AA_0 \rangle_B \leq_{HF} B \text{ and } \langle AB_0 \rangle_B \leq_{HF} B.$$ 
Hence, by Lemma~\ref{first_lemma_elem_sub} and the $\kappa^+$-homogeneity of $B$, the automorphism $g$ exists.
\end{proof}

	\begin{proof}[Proof of Theorem~\ref{th_strictly_stable}] This is by Theorems~\ref{unsuperstability}, \ref{cont_many_countable_open}~and~\ref{stability}.
\end{proof}

	\begin{lemma} Let $A, B \models T$, $\bar{a} \in A^n$ and $\bar{b} \in B^n$. Then $tp(\bar{a}/\emptyset; A) = tp(\bar{b}/\emptyset; B)$ if and only if there is an isomorphism $f: acl_A(\bar{a}) \rightarrow acl_B(\bar{b})$ such that $f(\bar{a}) = \bar{b}$.
\end{lemma}

	\begin{proof} The implication ``left-to-right'' is clear. Concerning the other implication, without loss of generality we can assume that $A$ and $B$ are $\omega_1$-saturated. By Lemma \ref{lemma_acl}, $A$ admits an HF-ordering over $\emptyset$ in which $acl(\bar{a})$ is closed, and analogously for $B$ and $\bar{b}$. Then we can play an Ehrenfeucht-Fra\"iss\'e game between $A$ and $B$ as in the proof of Lemma \ref{first_lemma_elem_sub}, starting from $f_{-1} = f$.
\end{proof}
	
	\begin{remark} If $A \subseteq B$,  $B \models T$, and $A$ contains a quadrangle, then $acl_B(A) \models T$.
\end{remark} 

	\begin{lemma}\label{lemma_prime_model} Let $A \subseteq B$ and suppose that $B \models T$ and $A$ contains a quadrangle. Then $acl_B(A)$ is a prime model over $A$.
\end{lemma}

	\begin{proof} Suppose that $f:A \rightarrow C$ is an elementary embedding, and let $C \preccurlyeq D$ be $|B|^+$-saturated. Then there is an elementary embedding $g: B \rightarrow D$ such that $f \subseteq g$. Then, $g \restriction acl_B(A)$ is an isomorphism onto $acl_D(f(A)) = acl_C(f(A))$ and, by Lemmas \ref{lemma_acl} and \ref{first_lemma_elem_sub},  $acl_C(f(A)) \preccurlyeq C$.
\end{proof}

	\begin{remark} Let $\mathfrak{M}$ be the monster model of $T$. Notice that given $A \subseteq B \subseteq \mathfrak{M}$, by Lemma \ref{lemma_acl}, we have that $acl_\mathfrak{M}(A) \leq_{HF} \mathfrak{M}$ and $acl_\mathfrak{M}(B) \leq_{HF} \mathfrak{M}$, and so, by Remark~\ref{remark_induced_HF}, we have that $acl_\mathfrak{M}(A) \leq_{HF} acl_\mathfrak{M}(B)$. This is relevant \mbox{for \ref{prop_pre_forking} and \ref{th_pre_forking}.}
\end{remark}

	\begin{convention} From now on, in this section, we will let $acl(A) = acl_{\mathfrak{M}}(A)$, where $\mathfrak{M}$ is the monster model of $T$, and $A \subseteq \mathfrak{M}$.
\end{convention}

	Recall the definition of $B \oplus_A C$ from Definition~\ref{def_canonical_amalgam} and recall Proposition~\ref{remark_can_amalgam}.

	\begin{proposition}\label{prop_pre_forking} Let $\mathfrak{M}$ be the monster model of $T$, $A, B, C \subseteq \mathfrak{M}$, and suppose that $acl(AC)$ contains a quadrangle. Then we can find $B' \subseteq \mathfrak{M}$ such that $acl(B'A)$ is isomorphic to $acl(BA)$ over $acl(A)$ and $acl(AB'C) \cong acl(AB') \oplus_{acl(A)} acl(AC)$.
\end{proposition}

	\begin{proof} This is clear.
\end{proof}

	\begin{lemma}\label{th_pre_forking} Let $\mathfrak{M}$ be the monster model of $T$, $A, B, C \subseteq \mathfrak{M}$, and suppose that $AB$ contains a quadrangle. If  $acl(ABC)$ is the {\em canonical amalgam} (cf. Def.~\ref{def_canonical_amalgam}) of $acl(AB)$ and $acl(AC)$ over $acl(A)$, then $C \pureindep[A] B$ (in the forking sense).
\end{lemma}

	\begin{proof} First of all, without loss of generality, we can assume that $A = acl(A)$,  $B = acl(AB)$, and $C = acl(AC)$. Suppose now that $acl(BC) = C \oplus_A B$ (cf. Definition~\ref{def_canonical_amalgam}) and that $B$ contains a quadrangle. We need to show that $C \pureindep[A] B$. For the sake of contradiction, suppose that $C \not\!\pureindep[A] B$. Choose $(B_i : i < \kappa)$, for $\kappa$ large enough, such that:
\begin{enumerate}[(a)]
	\item $B_0 = B$;
	\item $(B_i: i < \kappa)$ is a Morley sequence over $A$.
\end{enumerate}
Let $C'$ be a copy of $C$ over $A$ such that $C' \cap acl(\bigcup_{i < \kappa} B_i) = A$, then $C' \not\!\pureindep[A] B$ and so without loss of generality we can assume that $C' = C$. Let $D:= acl(\bigcup_{i < \kappa} B_i) $ and notice that by Lemmas \ref{first_lemma_elem_sub} and \ref{lemma_acl} we have that $D \preccurlyeq \mathfrak{M}$. Notice also that $A \leq_{HF} C$, $A \leq_{HF} D$, and $D \cap C = A$, and so we can consider the canonical amalgam $E:= D \oplus_A C \;\; \text{ (cf. Definition \ref{def_canonical_amalgam})}$. Observe now that: 
\begin{enumerate}[$(\star_1)$]
\item $(B_i : i < \kappa)$ is a Morley sequence over $A$ in the model $E \models T$.
\end{enumerate}
[Why? $D \leq_{HF} E$, since $A \leq_{HF} C$. Thus, as $D, E \models T$, by Lemma~\ref{first_lemma_elem_sub}, $D \preccurlyeq E$.]
\begin{enumerate}[$(\star_2)$]
\item In the model $E \models T$ we have that $C \not\!\pureindep[A] B_i$, for every $i < \kappa$.
\end{enumerate}
[Why? Since $B_i = acl(B_i) \preccurlyeq \mathfrak{M}$ and $D \preccurlyeq \mathfrak{M}$, we have that $B_i \preccurlyeq D$. And so, by Corollary~\ref{cor_elem_sbs}, $B_i \leq_{HF} D$. Hence, $\langle B_iC \rangle_E \leq_{HF} E$ (cf. Definition~\ref{subplane}), and so $\langle B_iC \rangle_E \preccurlyeq E$ (by Lemma~\ref{first_lemma_elem_sub}). Furthermore, clearly $\langle CB_i \rangle_E \cong C \oplus_A B_i$, and so there is an isomorphism $f_i: acl(CB) \cong_C \langle CB_i \rangle_E$ such that $f_i(B) = B_i$.]
\newline Hence, since $(B_i: i < \kappa)$ is Morley over $A$ in $E$, for every $i < \kappa$, we have that in $E$:
$$C \not\!\!\pureindep[B_{<i}] B_i,$$
where $B_{<i} = \bigcup_{j < i} B_j$. This contradicts the stability of $T$.
\end{proof}

	\begin{corollary}\label{cor_pre_forking} Let $\mathfrak{M}$ be the monster model of $T$, and $A, B, C \subseteq \mathfrak{M}$ and suppose that $A$ contains a quadrangle. Then $B \pureindep[A] C$ (in the forking sense) iff  $acl(ABC)$ is the {\em canonical amalgam} of $acl(AB)$ and $acl(AC)$ over $acl(A)$.
\end{corollary}

	\begin{proof} Since $B \pureindep[A] C$ iff $acl(B) \pureindep[acl(A)] acl(C)$, we may assume that $A = acl(A)$,  $B = acl(AB)$, and $C = acl(AC)$. Then the implication ``right-to-left'' is immediate from Lemma~\ref{th_pre_forking}, and the implication ``left-to-right'' follows from the implication ``right-to-left'', Prop.~\ref{prop_pre_forking} and the stationarity of $tp(B/A)$ (notice that $A \preccurlyeq \mathfrak{M}$).
\end{proof}

	\begin{proposition}\label{remark_acl} If $A \leq_{HF} B \models T$ and $\langle A \rangle_B = A$, then $acl_B(A) = A$.
\end{proposition}

	\begin{proof} The containment ``right-to-left'' is clear. Concerning the containment ``left-to-right'', suppose that there is $b \in acl_B(A) - A$. Now, for every $n < \omega$, let $(B_i : i < n)$ be such that:
\begin{enumerate}[(1)]
	\item $B_0 = B$;
	\item $B_i \cap B_j = A$, for $i < j < n$;
	\item $f_i: B \cong_A B_i$;
	\item $f_i(b) = b_i$
\end{enumerate} 
Now, let $n < \omega$ and consider the structure:
	$$B_0 \oplus_{A} B_1 \oplus_{A} \cdots \oplus_{A} B_{n-1} := D_n.$$
Then $B \leq_{HF} D_n \models T$ and so $B \preccurlyeq D_n$, from which it follows that $acl_B(A) = acl_{D_n}(A)$. Furthermore, without loss of generality, we can assume that $D_n \preccurlyeq \mathfrak{M}$. Hence, for every $i < n$, $tp(b_i/A) = tp(b/A)$. This leads to a contradiction.
\end{proof}

	\begin{theorem5} Let $\mathfrak{M}$ be the monster model of $T$, and $A, B, C \subseteq \mathfrak{M}$. Then $B \pureindep[A] C$~(in the forking sense) if and only if  $acl(ABC)$ is the {\em canonical amalgam} of $acl(AB)$ and $acl(AC)$ over $acl(A)$.
\end{theorem5}

	\begin{proof} Again we may assume that $A = acl(A)$, $B = acl(AB), $ and $C = acl(AC)$. 
\newline We prove the implication ``right-to-left''. To this extent, suppose that $acl(BC) = B \oplus_A C$, and let $A^+ \preccurlyeq \mathfrak{M}$ be such that $BC \subseteq A^+$ and $D \models T$ be such that $C \leq_{HF} D$ and $D \cap A^+ = C$. Let then $B^+:= A^+ \oplus_C D$, and notice that without loss of generality we may assume that $B^+ \preccurlyeq \mathfrak{M}$ (since $A^+ \preccurlyeq \mathfrak{M}$ and $A^+ \preccurlyeq B^+$, given that $A^+, B^+ \models T$ and $A^+ \leq_{HF} B^+$). Now, $\langle BD \rangle_{B^+} \cong B \oplus_A D$ and $\langle BD \rangle_{B^+} \leq_{HF} B^+$, and thus, by Proposition~\ref{remark_acl}, $acl(BD) = \langle BD \rangle_{B^+}$. Hence, by Corollary~\ref{cor_pre_forking}, $B \pureindep[A] D$, and thus, by Monotonicity,  we conclude that $B \pureindep[A] C$.
\newline We prove the implication ``left-to-right''. Let $A^+ \preccurlyeq \mathfrak{M}$ be such that $BC \subseteq A^+$ and $B^+ \models T$ be such that $A \leq_{HF} B^+ \models T$ and $B^+ \cap A^+ = A$. Let now $C^+ := A^+ \oplus_A B^+$. Since $A^+ \preccurlyeq \mathfrak{M}$ and $A^+ \preccurlyeq C^+$ (given that $A^+ \leq_{HF} C^+$) we can assume without loss of generality that $C^+ \preccurlyeq \mathfrak{M}$. Then, by Corollary~\ref{cor_pre_forking}, $B^+ \pureindep[A] A^+$. Thus, noticing that $C \pureindep[A] B$ and $C \pureindep[B] B^+$, and using Transitivity and Monotonicity, we conclude that $C \pureindep[B^+] B$. Hence, again by Corollary~\ref{cor_pre_forking}, we have $acl(BCB^+) \cong acl(B^+B) \oplus_{B^+} acl(CB^+)$. Also, we have $acl(BB^+) \cong B \oplus_A B^+$
(since $C^+ = A^+ \oplus_A B^+$ and $B \subseteq A^+$) and thus $acl(BB^+) = \langle BB^+ \rangle_{C^+}$. Analogously, we see that $C \oplus_A B^+ \cong acl(CB^+) = \langle CB^+ \rangle_{C^+}$. Hence, $H := \langle BC \rangle_{C^+} \cong B \oplus_A C$, and so, by Proposition~\ref{remark_acl}, it is enough to prove that $H \leq_{HF} acl(BCB^+)$. To this extent, notice that using what we observed above we have that:
$$\begin{array}{rcl}\label{equation}
H & \leq_{HF} & HB^+\\
  & \leq_{HF} & HB^+ \langle B B^+\rangle_{C^+}\\
  & \leq_{HF} & HB^+ \langle B B^+\rangle_{C^+}\langle C B^+\rangle_{C^+} \\
  & \leq_{HF} & HB^+ acl(BB^+) acl(CB^+) \\
  & \leq_{HF} & acl(BCB^+).
\end{array}$$
\end{proof}


	\begin{proposition}\label{no_disjoint_amalgamation} The class of finitely generated open projective planes is not closed under disjoint amalgamation, with respect to either the language $L=  (S_1, S_2, I)$ (considered in this paper) or the expanded language $L^+ =  (S_1, S_2, I, \wedge, \vee)$.
\end{proposition}

	\begin{proof} Let $A$ be a copy of $\pi^4$ generated by the quadrangle $\{ a, b, c, d \}$ and let:
	$$e = (a \vee c) \wedge (b \vee d), \; f = (a \vee b) \wedge (c \vee d), \;  g = (a \vee d) \wedge (b \vee c).$$
	Let $B = \langle d, e, f, g \rangle_A$. Then, using e.g. \cite[pg. 132]{sandler_coll} it can be seen that: 
	\begin{enumerate}[(i)]
	\item $b \notin B$ (since $\{ b, d, e, g \}$ is a generating quadrangle but $\{ d, e, f, g \}$ is not);
	\item $a \notin B$ (since $a \in B$ implies $(a \vee e) \wedge (d \vee f) = b \in B$);
	\item $c \notin B$ (since $c \in B$ implies $(c \vee f) \wedge (d \vee g) = a \in B$).
\end{enumerate}
Hence, $a, b, c \notin B$. Now, take an isomorphic copy $A'$ of $A$ over $B$ ($f: A' \cong_B A$), so that $A'$ is disjoint from $A$ over $B$, and let $f(a) = a', f(b) = b'$ and $f(c) = c'$. 
Let now $D$ be a partial plane which is the disjoint amalgam of $A$ and $A'$ over $B$. Then $D$ is not an open partial plane, since we can find an homomorphism (in the language $L=  (S_1, S_2, I)$) from the following non-open configuration into $D$:
\begin{enumerate}[(i')]
\item points: $a, b, c, a', b', c', d, e, f, g$;
\item lines: $fa'b', fab, fcdc', a'adg, bcg, b'c'g, aec, a'ec', b'bed$.
\end{enumerate}
\end{proof}

\end{document}